\documentclass[11pt]{article}

\usepackage{graphicx}
\usepackage{amssymb}

\usepackage{JASA_manu}

\usepackage{amsmath}
\usepackage{amsbsy}
\usepackage{enumerate}
\usepackage{bm}
\usepackage{comment}
\usepackage{hyperref}

\usepackage{subfigure}

\usepackage{amsthm}
\theoremstyle{definition}

\newtheorem{theorem}{Theorem}
\newtheorem{lemma}{Lemma}

\newtheorem{definition}{Definition}

\def\ind{\begin{picture}(9,8)
         \put(0,0){\line(1,0){9}}
         \put(3,0){\line(0,1){8}}
         \put(6,0){\line(0,1){8}}
         \end{picture}
        }

\newcommand{\Beta}{\text{Beta}}
\newcommand{\Binomial}{\text{Binomial}}
\newcommand{\Bern}{\text{Bern}}

\newcommand{\logit}{\text{logit}}

\newcommand{\Bin}{\text{Bin}}

\newcommand{\var}{\text{var}}
\newcommand{\cov}{\text{cov}}
\newcommand{\sumN}{\sum_{i=1}^N}

\newcommand{\CRD}{\textsc{CRD}}
\newcommand{\CRR}{\textsc{CRR}}
\newcommand{\COR}{\textsc{COR}}

\begin{document}

\title{\bf A Potential Tale of Two by Two Tables from Completely Randomized Experiments}
\author{Peng Ding and Tirthankar Dasgupta}
\date{}
\maketitle

\mbox{}
\vspace*{2in}
\begin{center}
\textbf{Authors' Footnote:}
\end{center}
Peng Ding is Ph.D. Candidate (E-mail: \texttt{pengding@fas.harvard.edu}), and Tirthankar Dasgupta is Associate Professor (E-mail: \texttt{dasgupta@stat.harvard.edu}), Department of Statistics, Harvard University, One Oxford Street, Cambridge 02138, MA.
The authors thank Professor Donald Rubin for pointing out Cornfield (1959) and Rubin (2012)'s discussion about experimental and observational studies, and also
thank the participants in Professor Rubin's seminar ``Research Topics in Missing Data, Matching and Causality'' at Harvard in Fall 2013, for their helpful comments.
We are also grateful to the reviewers and Associate Editor, the
comments from whom greatly improved the presentation of our paper.

\newpage
\begin{center}
\textbf{Abstract}
\end{center}

Causal inference in completely randomized treatment-control studies with binary outcomes is discussed from Fisherian, Neymanian and Bayesian perspectives, using the potential outcomes framework. A randomization-based justification of Fisher's exact test is provided. Arguing that the crucial assumption of constant causal effect is often unrealistic, and holds only for extreme cases, some new asymptotic and Bayesian inferential procedures are proposed. The proposed procedures exploit the intrinsic non-additivity of unit-level causal effects, can be applied to linear and non-linear estimands, and dominate the existing methods, as verified theoretically and also through simulation studies.

\bigskip

\section{\bf Introduction}

The theory of causal inference from randomized treatment-control studies using the potential outcomes model has been well-developed over the past five decades and has been applied extensively to randomized experiments in the medical, behavioral and social sciences. The first formal notation for potential outcomes was introduced by Neyman (1923) in the development of randomization-based inference, and subsequently used by several researchers including Kempthorne (1952) and Cox (1958) for drawing causal inference from randomized experiments. The concept was formalized and extended by Rubin (1974, 1977, 1978) for other forms of causal inference from randomized experiments \emph{and} observational studies, and exposition of this transition appears in Rubin (2010). The three broad approaches to causal inference under the potential outcomes model are Fisherian, Neymanian and Bayesian.

The most common finite-population estimand in most causal inference problems is the average causal effect, defined as the finite-population average of unit-level causal effects. Since Neyman's (1923) seminal work, additivity of unit-level treatment effects (or its lack thereof) and its influence on the inference for the average causal effect has been investigated thoroughly for continuous outcomes. In comparison, few researchers (e.g., Copas 1973) have studied this problem for binary outcomes, in which the potential and observed outcomes can be summarized in the form of $2 \times 2$ contingency tables. In this paper, we provide a characterization of additivity based on the $2 \times 2$ table of potential outcomes, and use it to (i) justify Fisher's exact test from a randomization perspective, and (ii) propose an estimator of the variance of the average causal effect for binary outcomes that uniformly dominates the Neymanian variance estimator. As advocated by Rubin (1978), we also propose a Bayesian strategy for drawing inference about the average causal effect using the missing data perspective. Such a strategy is dependent on the assumptions related to model additivity, or more specifically, the nature and strength of the association between potential outcomes. We propose a novel sensitivity analysis which should help a practitioner understand how the analysis results might change if the assumptions are violated.

Apart from the average causal effect, other popular estimands for binary outcomes are the log of the causal risk ratio and the log of the causal odds ratio. Although of great practical interest, to the best of our knowledge, estimators of these causal measures have not been studied carefully from the Neymanian perspective, because unlike the average causal effect, non-linearity of these estimands and their estimators make exact variance calculations intractable. We circumvent this problem by taking an asymptotic perspective. By deriving asymptotic expressions for variances of these estimators, we explore the adequacy of the widespread practice of drawing statistical inference for such causal estimands on the basis of independent Binomial models, and propose improved methods that are justified by randomization. We conduct simulation studies under different settings to demonstrate the effectiveness of the proposed methods and also illustrate their application to a recent randomized controlled trial.

The paper is organized as follows. In the following section, we define the potential outcomes, the finite population estimands, the assignment mechanism and the observed outcomes. In Section 3, we discuss the Fisherian and Neymanian forms of inference for $2 \times 2$ tables. In Section 4, we propose a Bayesian framework for causal inference, explore its frequentists' properties and also propose a methodology for sensitivity analysis to assess the effect of violation of assumptions regarding additivity (or its lack thereof) on the inference. Causal inference for non-linear estimands is discussed in Section 5. A detailed simulation study is conducted in Section 6 to compare the different methods of inference and to demonstrate the superiority of the proposed methodology. Application of the proposed methodology to randomized experiments with binary outcomes is demonstrated with a real-life example in Section 7. Some concluding remarks are presented in Section 8.
Some technical details are in Appendix A, and the proofs, additional simulation studies and more details of the application are in the online Supplementary Materials.

\section{\bf Potential outcomes, estimands, and the observed data}
\label{sec::potential-outcomes-observed-data}

The evolution of the potential outcomes framework was motivated by the need for a clear separation between the object of interest (often referred to as the ``Science'') and what researchers do to learn about the Science (e.g., randomly assign treatments to units). We assume a finite population of $N$ experimental units that are exposed to a binary treatment $W$ and yield a binary response $Y$. Under the Stable Unit Treatment Value Assumption (Cox 1958; Rubin 1980), we define $Y_i(t)$ as the potential outcome for individual $i$ when exposed to treatment $t$ ($t=1$ and $t=0$ often refer to treatment and control, respectively). The $N\times 2$ matrix of the potential outcomes $\{(Y_i(1), Y_i(0)):i=1,\ldots,N\}$ is typically referred to as the Science (Rubin 2005). Because the response $Y$ is binary, the information contained in the Science can be condensed into a $2 \times 2$ table as shown in Table \ref{tb::potential}.

\begin{table}[ht]
\caption{``Science Table'' of the Potential Outcomes }\label{tb::potential}
\centering
\begin{tabular}{|c|cc|c|}
\hline
  &  $Y(0)=1$  & $Y(0)=0$  & row sum\\
\hline
$Y(1)=1$ & $N_{11} $  & $N_{10}$  & $N_{1+}$  \\
$Y(1)=0$ & $N_{01} $  & $N_{00} $ & $N_{0+}$ \\
\hline
column sum & $N_{+1}$ & $N_{+0}$ & $N$\\
\hline
\end{tabular}
\end{table}

\subsection{\bf Finite-population causal estimands and uniformity of unit-level causal effects} \label{ss::estimands}

Having defined the so-called Science, we now proceed to the definition of causal estimands. A unit-level causal effect is defined as a contrast between the potential outcomes under the treatment and the control, for example, $\tau_i = Y_i(1) - Y_i(0)$.
We define the finite population
average causal effect as
$$
\tau = \frac{1}{N} \sumN \tau_i  = p_1- p_0 ,
$$
where $p_t= \sumN Y_i(t)/N$ is the finite population average of $Y(t)$ for $t=0,1$. For binary outcomes, the average causal effect is also called the causal risk difference ($\CRD$). From Table \ref{tb::potential}, it follows that $$ \tau = N_{1+}/N - N_{+1}/N = \left( N_{10} - N_{01} \right) /N. $$ A measure of uniformity (or its lack, thereof) of the unit-level causal effects is the finite population variance of the individual causal effect $\tau_i$, given by
\begin{equation}
S_{\tau}^2 =  \frac{1}{N-1} \sumN \left(   \tau_i -  \tau  \right)^2 
= \frac{1}{N(N-1)}  \left\{   (N_{10} + N_{01})(N_{11}+N_{00}) + 4N_{10}N_{01}  \right\}.
 \label{eq::var_tau}
\end{equation}
Note that $S_{\tau}^2$ can also be represented as
\begin{equation}
S_{\tau}^2 = S_1^2 + S_0^2 - 2 S_{10}, \label{eq:var_tau2}
\end{equation}
where $S_t^2  =    \sumN \{ Y_i(t) - p_t  \}^2/(N-1)$ is the finite population variance of the potential outcome $Y_i(t)$,
and $S_{10} = \sumN \{ Y_i(1) - p_1 \} \{ Y_i(0) - p_0 \}/(N - 1) = (N_{11}N_{00} - N_{10}N_{01}) / \{ N(N-1)\}$ is the finite population covariance between $Y_i(1)$ and $Y_i(0)$.

Note that constant causal effect or additivity of unit-level causal effects implies that $S_1^2=S_0^2, S_{10}=S_1S_0$, and the uniformity measure $S_\tau^2=0$.
Copas (1973) considered a representation of the potential outcomes similar to that in Table \ref{tb::potential}, and defined parameters $\alpha = \tau$ as the treatment effect and $\beta = (N_{10} + N_{01})/N$ as a measure of ``the differential effect.'' However, we feel that $\beta$, which essentially equals $(\sum_{i=1}^N \tau_i^2 )/N$, is not an adequate representation of the differential effect, because it does not reduce to zero when all unit-level causal effects are equal to $1$ or $-1$. To discuss this aspect further, we consider the case of \emph{strict additivity} of treatment effects, where $\tau_i = \tau$ for all $i = 1, \ldots, N$, and summarize its impact on the Science and its summary measures $\tau$, $S_{\tau}^2$ and $\beta$ in Table \ref{tb::additivity}.

\begin{table}[ht]
\caption{Effect of additivity on the Science} \label{tb::additivity}
\centering
\begin{tabular}{|c|c|c|c|c|} \hline
$\tau (= \tau_i)$ & Entries of Table  \ref{tb::potential} & $\tau = \alpha$ & $S_{\tau}^2$ & $\beta$ \\ \hline
$1$              &  $N_{11}=N_{01}=N_{00}=0$, $N_{10}=N$ &    $1$            &  $0$           &   $1$     \\ \hline
$-1$             &  $N_{11}=N_{10}=N_{00}=0$, $N_{01}=N$ &    $-1$           &  $0$           &   $1$     \\ \hline
$0$              &  $N_{10}=N_{01}=0$, $N_{00}+N_{11}=N$  &    $0$            &  $0$           &   $0$     \\ \hline
\end{tabular}
\end{table}

Note that the last row of Table \ref{tb::additivity} represents a special case of additivity, with zero treatment effect for each unit. Such a hypothesis about the Science case is referred to as Fisher's sharp null hypothesis of no treatment effect, and forms the basis of the ``Fisherian'' inference described in Section \ref{sec::fisher}.

To sum up our discussion on the degree of uniformity of treatment effects, we define another condition referred to as \emph{monotonicity} (Angrist et al. 1996).

\begin{definition} \label{def::monotonicity}
The Science table is said to satisfy the monotonicity condition if either of the following two conditions hold:
(i) $Y_i(1) \ge Y_i(0)$ for all $i$ (or equivalently $N_{01}=0$), (ii) $Y_i(1) \le Y_i(0)$ for all $i$ (or equivalently $N_{10}=0$).
\end{definition}

Under monotonicity, we have $ \tau = \alpha = \beta = N_{10}/N, \;\ S_{\tau}^2 = N_{10} N_{11} / \{ N(N-1) \}$ if (i) holds, and $\tau = \alpha = -N_{01}/N, \;\ \beta = N_{01}/N, \;\ S_{\tau}^2 = N_{01} N_{11} / \{ N(N-1) \}$ if (ii) holds. We also note that any one of the three additivity conditions as described in Table \ref{tb::additivity} implies at least one of the monotonicity assumptions. We shall discuss the impact of strict additivity and monotonicity on the inference for $\tau$ in Section \ref{sec::neyman}.

\subsection{\bf Treatment assignment and the observed data} \label{ss::observed_outcomes}

We consider a completely randomized treatment assignment in which $N_1$ and $N_0$ units receive treatments 1 and 0 respectively. Let $\bm{W} = (W_1, \ldots, W_N)$ be the vector of treatment assignments and let $\bm{w}=(w_1,\ldots,w_N)$ be a realization of $\bm{W}$. Then, a completely randomized experiment satisfies $P(\bm{W} = \bm{w}) = N_1!N_0!/N!$ if $\sumN w_i = N_1$ and $P(\bm{W} = \bm{w})$ otherwise.
The observed outcomes are deterministic functions of both the treatment and the potential outcomes, since $ Y_i^{\text{obs}} = W_iY_i(1) + (1-W_i) Y_i(0) $.  Let $\bm{Y}^{\text{obs}}  = (Y_1^{\text{obs}}, \cdots, Y_N^{\text{obs}})$ be the vector of the observed outcomes. Since the treatment and the outcome are both binary, the observed data form a $2 \times 2 $ contingency table as shown in Table \ref{tb::obs}. The row sums in Table \ref{tb::obs}, $(N_1, N_0)$, are the numbers of individuals receiving treatment and control, and the column sums, $(n_{+1}, n_{+0})$, are the
number of individuals with outcomes $1$ and $0$, respectively.

\begin{table}[ht]
\centering
\caption{Summary of the Observed Data}\label{tb::obs}
\begin{tabular}{|c|cc|c|}
\hline
  &  $Y^{\text{obs}}=1$  & $Y^{\text{obs}}=0$ & row sum \\
\hline
$W=1$ & $n_{11} $  & $n_{10}$ & $N_1$\\
$W=0$ & $n_{01} $  & $n_{00} $ & $N_0$\\
\hline
column sum & $n_{+1}$ & $n_{+0}$ & $N$  \\
\hline
\end{tabular}
\end{table}

We conclude this section by emphasizing that the fundamental problem of causal inference is the missingness of one element of each pair $(Y_i(1), Y_i(0))$. Consequently, the key idea is to infer about the entries of Table \ref{tb::potential} (and the estimands that are functions of these unknown entries) using those of Table \ref{tb::obs} and the distribution of these entries under randomization.

\section{\bf Fisherian and Neymanian approaches to inference} \label{sec::fisher_neyman}

In this section, the potential outcomes of the finite population are assumed to be fixed numbers, and the randomness in the observed outcomes comes only from randomization of the treatment assignment (Neyman 1923; Rubin 1990). We discuss two forms of finite-population inference --- Fisherian and Neymanian --- under this set-up.
Fisher's form of randomization-based inference focuses on assessing the sharp null hypothesis of no treatment effect using the randomization distribution of a test statistic, which is obtained by imputing the missing outcomes under the sharp null.
Neyman's form of randomization-based inference can be viewed as drawing inferences by evaluating the expectations of statistics over the distribution induced by the assignment mechanism in order to calculate a confidence interval for the typical causal effect. Using asymptotic results is one way of achieving this.
In the following subsection (Section 3.1), we briefly discuss the Fisher randomization test and establish its connection to Fisher's exact test. In Section 3.2, we discuss Neymanian inference and propose an improvement over the traditional Neymanian estimator.

\subsection{\bf Fisherian randomization test and its connection to Fisher's exact test} \label{sec::fisher}

According to Fisher (1935a), randomization yields ``a reasoned basis for inference,'' and it allows for testing the sharp null hypothesis of zero individual causal effect, i.e., $Y_i(1) = Y_i(0)$ for $i = 1, \ldots, N$, characterized by the last row of Table \ref{tb::additivity}. Such a null hypothesis permits imputation of all the missing potential outcomes. Although in principle any test statistic can be used, the most natural one is $\widehat{\tau}  =  \widehat{p}_1 - \widehat{p}_0$, where $\widehat{p}_1 = \sumN W_i Y_i^{\text{obs}} /N_1 = n_{11}/N_1$, and $\widehat{p}_0 =  \sumN (1 - W_i)Y_i^{\text{obs}} / N_0 = n_{01}/N_0.$ The test statistic $\widehat{\tau}$ is a function of both the treatment assignment and the observed outcomes. Under the sharp null hypothesis, the randomness of $ \widehat{\tau} $ comes only from the randomization of the treatment assignment $\bm{W}$. The $p$-value under the sharp null is a measure of the extremeness of the observed value of the test statistic with respect to its randomization distribution under the sharp null. For a two-sided test, the $p$-value is typically defined as the proportion of values of $|\widehat{\tau}|$ generated under all possible randomizations that exceed its observed value $ |\widehat{\tau}^{\text{obs}}|$. In general, the null distribution of $ \widehat{\tau} $ and the $p$-value can either be calculated exactly, or approximated by Monte Carlo.

However, we can obtain the ``exact'' distribution of the randomization test statistic for a binary outcome.
In Table \ref{tb::obs}, the margins $N_1$ and $ N_0$ are fixed by design. Under the sharp null hypothesis, the margins $n_{+1}$ and $n_{+0}$ represent the number of observations with potential outcomes $Y_i(1)=Y_i(0)=1$ and $Y_i(1)=Y_i(0)=0$, respectively, and are equal to $N_{11}$ and $N_{00}$ in Table \ref{tb::potential}. It follows that
\begin{eqnarray}
\label{eq::tau}
\widehat{\tau} =   {n_{11}  \over N_1} - { n_{01} \over N_0} = { n_{11} \over N_1} - {  n_{+1} -  n_{11}    \over N_0}
= \frac{N}{N_1 N_0} n_{11} - \frac{  N_{11}  }{ N_0},
\end{eqnarray}
i.e., the test statistic $\widehat{\tau}$ is a monotone function of $n_{11}$. Therefore, the rejection region based on $\widehat{\tau}$ is equivalent to the rejection region based on $n_{11}$, which has the usual Hypergeometric null distribution of the exact test for a two by two contingency table.

More interestingly, any randomization test using statistics other than $\widehat{\tau}$ is also equivalent to the test based on $n_{11}$, since any test statistic is a function of $n_{11}$ under Fisher's sharp null hypothesis.
Numerically, the test has exactly the same form as Fisher's exact test, although the two tests were originally derived based on completely different statistical reasonings.
In observational studies under Multinomial or independent Binomial sampling, Fisher (1935b) justified his exact test for association as a conditional test, by arguing that the marginal counts are nearly ancillary.
However, it turns out that the marginal counts contain some information about the association (Chernoff 2004), and they are not ancillary.
Here, we give a justification of the validity for Fisher's exact test based on randomization, if the data truly come from a completely randomized experiment.
For more discussion about the hypothesis testing issue, see Berkson (1978), Yates (1984) and Chernoff (2004) for observational studies, and Ding (2014) for randomized experiments.

\subsection{\bf Neymanian inference for the average causal effect} \label{sec::neyman}

Neyman (1923) showed that $\widehat{\tau} = \widehat{p}_1 - \widehat{p}_0$ is unbiased for $\tau$, with the sampling variance
\begin{eqnarray}
\var(\widehat{\tau})
= \frac{N_0}{N_1N} S_1^2 + \frac{N_1}{N_0N}S_0^2 + \frac{2}{N}S_{10}
= \frac{S_1^2}{N_1} + \frac{S_0^2}{N_0} - \frac{S_{\tau}^2}{N}, \label{eq::var}
\end{eqnarray}
where $S_{\tau}^2$, $S_1^2$ and $S_0^2$ are defined in Section \ref{ss::estimands}. The proof can be found in Neyman (1923) or directly from Lemma A.2 in Appendix B.
Since the third term in (\ref{eq::var}), $ S_{\tau}^2 / N $, depends on the joint distribution of the potential outcomes, it is not identifiable from the observed data without further assumptions.
Because of this difficulty, Neyman (1923) proposed a ``conservative'' estimator for $\var(\widehat{\tau})$, defined as
\begin{eqnarray}
\widehat{V}_{Neyman} = {  s_1^2 \over N_1} +  { s_0^2 \over N_0},
\end{eqnarray}
where $ s_1^2 =   \sum_{W_i=1} (Y_i^{\text{obs}} - \widehat{p}_1)^2/(N_1 -1)$ and $ s_0^2 =  \sum_{W_i=0} (Y_i^{\text{obs}} - \widehat{p}_0 )^2/(N_0 -1)$ are the sample variances of the observed outcomes under the treatment and the control, respectively. For binary outcomes, the variance estimator can be simplified as
\begin{eqnarray}
\widehat{V}_{Neyman} &=&  \frac{1}{N_1(N_1 -1)} \left(n_{11} - N_1 \frac{n_{11}^2}{N_1^2}  \right)
+  \frac{1}{N_0(N_0 -1)} \left(n_{01} - N_0 \frac{n_{01}^2}{N_0^2}  \right) \nonumber  \\
&=& \frac{\widehat{p}_1 (1 - \widehat{p}_1)}{N_1 - 1} + \frac{\widehat{p}_0 (1 - \widehat{p}_0)}{N_0 - 1}. \label{eq::var-neyman}
\end{eqnarray}
As we will discuss later in Section \ref{subsec::binomial}, (\ref{eq::var-neyman}) is very close to the standard formula for the variance of the difference of sample proportions, except for the fact that the coefficient denominator in the latter are $N_w$ instead of $N_w-1$ for $w=0,1.$

The variance estimator $\widehat{V}_{Neyman}$ is ``conservative'' in the sense that it only unbiasedly estimates the first two terms of (\ref{eq::var}), $S_1^2/N_1 + S_0^2/N_0$, and therefore $E(\widehat{V}_{Neyman})  \geq \var(\widehat{\tau}),$ a fact pointed out by several authors, e.g., Gadbury (2001), who provided an expression for the bias of the estimator. The variance estimator $\widehat{V}_{Neyman}$ is unbiased for the true variance if and only if the individual causal effects are constant ($\tau_i=\tau$) or, equivalently, the conditions in Table \ref{tb::additivity} are satisfied. Neyman (1923)'s constant causal effect assumption is equivalent to using $0$ as a lower bound for $S_\tau^2$, which is not sharp for binary outcomes.
Consequently, Neyman's ``conservative'' variance estimator can be improved for binary outcomes, even if the potential outcomes are not strictly additive. The following result gives the sharp lower bound for $S_\tau^2/N$ in terms of $\tau$.

\begin{theorem}
\label{thm::bound-tau}
A lower bound for $S_\tau^2/N$ is
\begin{eqnarray}\label{eq::bound-s-tau}
\frac{S_\tau^2}{N} \geq \frac{|\tau| (1 - |\tau|)}{N-1},
\end{eqnarray}
and equality holds if and only if the potential outcomes satisfy the monotonicity condition as stated in Definition \ref{def::monotonicity}.
\end{theorem}

Theorem \ref{thm::bound-tau} implies that the monotonicity assumptions are the most ``conservative'' cases for variance estimation.  As shown in the proof of Theorem \ref{thm::bound-tau}, the lower bound (\ref{eq::bound-s-tau}) for $S_\tau^2$ is obtained via an optimization approach, which minimizes $S_\tau^2$ under the constraints of the marginal distributions $p_1$ and $p_0$. Therefore, the lower bound in (\ref{eq::bound-s-tau}) is ``sharp'', in the sense that it cannot be uniformly improved without further assumptions.

The lower bound for $S_\tau^2/N$ allows us to define the following estimator, which is an improvement over Neyman's variance estimator given by (\ref{eq::var-neyman}):
\begin{eqnarray}
\label{eq::improve-neyman-CRD}
\widehat{V}_{Neyman}^c =
\frac{\widehat{p}_1 (1 - \widehat{p}_1)}{N_1 - 1} + \frac{\widehat{p}_0 (1 - \widehat{p}_0)}{N_0 - 1}
- \frac{  |\widehat{\tau}| (1 - | \widehat{\tau} | )  }{N-1}.
\end{eqnarray}
This estimator cannot be larger than $ \widehat{V}_{Neyman}$, and is also an improvement of the variance estimator given in Robins (1988).
If $\tau\in \{1, -1, 0 \}$, i.e., if $S^2_{\tau}=0$, we have $|\tau|(1 - |\tau |) = 0$, and with large sample size, the adjusting term $  |\widehat{\tau}| (1 - | \widehat{\tau} | ) / (N-1) $ is of higher order relative to the two leading terms of the variance estimator (\ref{eq::improve-neyman-CRD}). Therefore, if $S^2_{\tau}=0$, then asymptotically, the adjusting term does not hurt, and $\widehat{V}_{Neyman}^c $ and $ \widehat{V}_{Neyman}$ are equivalent.
However, for small samples, we may under-estimate the true sampling variance due to the positive adjusting term $|\tau|(1 - |\tau |)/(N-1)$.
We will investigate this finite sample issue further in the simulation studies.
If the true average causal effect is not $-1, 0,$ or $1$, i.e., $S^2_{\tau} \ne 0$ the correction term $|\widehat{\tau}|(1-|\widehat{\tau}|)/(N-1)$ in the variance estimator cannot be asymptotically neglected, and
the ``adjusted'' variance estimator will improve Neyman (1923)'s original variance estimator.
For example, if we observed a two by two table with cell counts $(n_{11}^\text{obs}, n_{10}^\text{obs}, n_{01}^\text{obs}, n_{00}^\text{obs}) = (15,5,5,15)$, then we have $\widehat{p}_1 =0.75, \widehat{p}_0=0.25, \widehat{V}_{Neyman} = 0.020,$ and $ \widehat{V}_{Neyman}^c = 0.013$, with the latter variance estimator $32.48\%$ smaller than the former one.

\section{\bf Bayesian causal inference for binary outcomes}
\label{sec::bayes}
In this section, we adopt the Bayesian causal inference framework advocated by Rubin (1978).
We assume that all the potential outcomes are drawn from a hypothetical super-population, while we are still interested in making inference on the finite population average causal effect $\tau$.
Similar to Neymanian randomization inference, the association between the potential outcomes is also crucial for our Bayesian causal inference.
We first propose a Bayesian procedure based on a simple model with independent potential outcomes, and discuss its frequentists' repeated sampling property (Rubin 1984).
We then propose a sensitivity analysis procedure to investigate the impact of departures from the independence assumption on Bayesian inference.

\subsection{\bf Independent potential outcomes}
\label{sec::ind}

Assume that we have the following model with independent potential outcomes:
$$
Y_i(1)\sim \Bern(\pi_{1+}), \quad
Y_i(0)\sim \Bern(\pi_{+1}), \quad
Y_i(1)\ind Y_i(0), \quad
i=1, \cdots, N,
$$
where ``$\ind$'' denotes independence. The notation $(\pi_{1+}, \pi_{+1})$ is chosen to be coherent with the marginal probabilities in Table \ref{tb::model-potential} to be discussed later.
We will relax the independence assumption in Section \ref{sec::BSA}.
We postulate the following priors $\pi_{1+}\sim \Beta(\alpha_1, \beta_1), \pi_{+1}\sim \Beta(\alpha_0, \beta_0)$, and assume that they are independent {\it a priori}.

Since the treatment assignment mechanism is ignorable (Rubin 1978) in completely randomized experiments,
the joint posterior distribution of $\pi_{1+}$ and $\pi_{+1}$ is
\begin{eqnarray} \label{eq::post}
f(\pi_{1+}, \pi_{+1}\mid \bm{W}, \bm{Y}^{\text{obs}}) \propto
\pi_{1+}^{\alpha_1 - 1} (1 - \pi_{1+})^{\beta_1 - 1}  \pi_{+1}^{\alpha_0 - 1} (1 - \pi_{+1})^{\beta_0 - 1}
\pi_{1 + }^{n_{11}} (1 - \pi_{1+})^{n_{10}}  \pi_{+1}^{n_{01}} (1 - \pi_{+1})^{n_{00}},
\end{eqnarray}
or equivalently,
$\pi_{1+}|\bm{W}, \bm{Y}^{\text{obs}} \sim \Beta(n_{11} + \alpha_1, n_{10} + \beta_1), \pi_{+1}|\bm{W}, \bm{Y}^{\text{obs}} \sim \Beta(n_{01}+\alpha_0, n_{00} + \beta_0)$, and they are independent {\it a posteriori}.
After obtaining the posterior distribution of $(\pi_{1+}, \pi_{+1})$, we can impute all the missing potential outcomes, conditioning on $(\pi_{1+}, \pi_{+1})$.
If $W_i=1,$ we impute $Y_i(0)|\bm{W}, \bm{Y}^{\text{obs}}, \pi_{+1} \sim \Bern(\pi_{+1})$; and if $W_i = 0$, we impute $Y_i(1)|\bm{W}, \bm{Y}^{\text{obs}}, \pi_{1+} \sim \Bern(\pi_{1+})$.
Therefore, the posterior distribution of $\tau$ conditioning on $\pi_{1+}$ and $\pi_{+1}$ is
\begin{eqnarray}
\label{eq::post-tau}
\tau|\bm{W}, \bm{Y}^{\text{obs}}, \pi_{1+}, \pi_{+1} \sim  \frac{ n_{11} +B_0 - n_{01} -B_1 }{N},
\end{eqnarray}
where $B_1\sim \Binomial(N_1, \pi_{+1}), B_0 \sim \Binomial(N_0, \pi_{1+})$, and they are independent.
The description above also illustrates a Monte Carlo strategy for simulating the posterior distribution of $\tau.$
For theoretical comparison with Neymanian inference, we can also obtain the posterior mean and variance of $\tau$ as follows.
We give the exact formulae for posterior mean and variance in Appendix C online, and here for simplicity we give approximate formulae.

\begin{theorem}
\label{thm::post-mean-var}
Assume that the prior pseudo counts $(\alpha_0, \beta_0, \alpha_1, \beta_1)$ are small compared to $n_{ij}$'s.
The posterior mean of $\tau$ is
\begin{eqnarray*}
E(\tau \mid \bm{W}, \bm{Y}^{\text{obs}}) \approx  \widehat{\tau},
\end{eqnarray*}
and the posterior variance of $\tau$ is
\begin{eqnarray}\label{eq::post_var}
\var(\tau \mid \bm{W}, \bm{Y}^{\text{obs}})
\approx   \frac{N_0}{N}  \frac{\widehat{p}_1 (1 - \widehat{p}_1)}{N_1 - 1} + \frac{N_1 }{N} \frac{\widehat{p}_0 (1 - \widehat{p}_0)}{N_ 0 -1 } .
\end{eqnarray}
\end{theorem}

From Theorem \ref{thm::post-mean-var}, we can see that the posterior variance of $\tau$ is smaller than Neyman's variance estimator.
These variances are different because
Neyman (1923) assumed perfect correlation between the potential outcomes, while the Bayesian model assumes independence between the potential outcomes.
As shown in (\ref{eq::var}), the assumption that $S_\tau^2 = 0$ is the worst case for the variance of $\var(\widehat{\tau})$, and Neyman (1923) adopted this as the most ``conservative'' estimator for the true variance.

\subsection{\bf Frequency evaluation of the Bayesian procedure under independence}

Going back to the finite population perspective, the sampling distribution of $\widehat{\tau}$ depends on the finite population covariance between $Y_i(1)$ and $Y_i(0)$, as shown in (\ref{eq::var}).
Assuming independence between $Y_i(1)$ and $Y_i(0)$, we have $S_{10} = 0$, and (\ref{eq::var}) becomes
$$
\var(\widehat{\tau}) = \frac{N_0}{N_1 N} S_1^2 + \frac{N_1}{N_0 N} S_0^2.
$$
The variance of $\widehat{\tau}$ can be unbiasedly estimated by
\begin{eqnarray}
\label{eq::var-ind}
\widehat{V}_{ind} =  \frac{N_0}{N_1 N} s_1^2 + \frac{N_1}{N_0 N} s_0^2
= \frac{N_0}{N}  \frac{\widehat{p}_1 (1 - \widehat{p}_1)}{N_1 - 1} + \frac{N_1 }{N} \frac{\widehat{p}_0 (1 - \widehat{p}_0)}{N_ 0 - 1 } .
\end{eqnarray}
The estimator of the sampling variance of $\widehat{\tau}$ in (\ref{eq::var-ind}) and the approximated posterior variance of $\tau$ in (\ref{eq::post_var}) under independence are the same.

Therefore, the Bayesian credible interval under independence will have a correct asymptotic coverage property, if the finite population covariance of the potential outcomes is zero.
However,
if the finite population covariance between $Y_i(1)$ and $Y_i(0)$, $S_{10}$, is negative, we have
$$
\var(\widehat{\tau}) < \frac{N_0}{N_1 N} S_1^2 + \frac{N_1}{N_0 N} S_0^2
$$
according to (\ref{eq::var}), which implies that the Bayesian credible interval will over-cover the truth over repeated sampling.
If the finite population covariance between $Y_i(1)$ and $Y_i(0)$, $S_{10}$, is positive, the Bayesian credible interval may not have a correct frequentists' coverage property.

\subsection{\bf Bayesian sensitivity analysis}
\label{sec::BSA}

The independence between potential outcomes may not be plausible even conditionally on observed covariates.
In particular, if the potential outcomes are positively correlated, the Bayesian credible interval may not have a correct frequentists' coverage property.
However, the observed data provide no information about the association between the two potential outcomes,
since they are never jointly observed.
Therefore, we propose a sensitivity analysis approach for the Bayesian model discussed above.

\begin{table}[ht]
\caption{Model of the Potential Outcomes}\label{tb::model-potential}
\centering
\begin{tabular}{|c|cc|c|}
\hline
  &  $Y(0)=1$  & $Y(0)=0$ &  row sum \\
\hline
$Y(1)=1$ & $\pi_{11} $  & $\pi_{10}$ & $\pi_{1+}$\\
$Y(1)=0$ & $\pi_{01} $  & $\pi_{00} $ & $1-\pi_{1+}$\\
\hline
column sum & $\pi_{+1}$ & $1-\pi_{+1}$ & $1$  \\
\hline
\end{tabular}
\end{table}

The joint distribution of $( Y_i(1), Y_i(0)   )$ follows a Multinomial distribution with parameters $(\pi_{11}, \pi_{10}, \pi_{01}, \pi_{00})$ as shown in Table \ref{tb::model-potential}, which can be equivalently characterized by the marginal distributions $(\pi_{1+}, \pi_{+1})$ and an association parameter.
We propose a new characterization of association between the potential outcomes in terms of the sensitivity parameter:
$$
\gamma =  \frac{P\{ Y(1)=1\mid Y(0) = 1 \} }{ P\{ Y(1) = 1\mid Y(0)=0\} }  = \frac{ \pi_{11}}{\pi_{10}}  \frac{1-\pi_{+1}}{\pi_{+1}} \in (0, \infty) .
$$
When the potential outcomes are independent, we have $\gamma =1$; when $\pi_{11}\rightarrow  0$, we have $\gamma \rightarrow 0$; when $\pi_{10}\rightarrow 0$, we have $\gamma \rightarrow \infty.$
In practice, we propose varying our sensitivity parameter $\gamma$ over a wide range of values, and performing Bayesian inference at each fixed value of $\gamma$.

There is a one-to-one mapping between $(\pi_{11}, \pi_{10}, \pi_{01}, \pi_{00})$ and $(\pi_{1+}, \pi_{+1}, \gamma)$, and thus the cell probabilities $\pi_{jk}$'s can be expressed as
\begin{eqnarray}
\pi_{11} = \frac{\gamma \pi_{1+} \pi_{+1}}{ 1 - \pi_{+1} + \gamma \pi_{+1} },&&
\pi_{10} =  \frac{  \pi_{1+}(1- \pi_{+1})   }{ 1 - \pi_{+1} + \gamma \pi_{+1} },\label{eq::cell1}\\
\pi_{01} = \pi_{+1} -  \frac{\gamma \pi_{1+} \pi_{+1}}{ 1 - \pi_{+1} + \gamma \pi_{+1} },&&
\pi_{00} = 1-\pi_{+1}-\pi_{1+}+\pi_{11}.\label{eq::cell2}
\end{eqnarray}
Since all the cell probabilities are within the interval $[0, 1]$, the equations in (\ref{eq::cell1}) and (\ref{eq::cell2}) impose the following restrictions on $(\pi_{1+}, \pi_{+1}, \gamma)$:
\begin{eqnarray}
\gamma (\pi_{1+} - \pi_{+1}) \leq 1 - \pi_{+1}, \quad
\gamma \pi_{+1} > \pi_{1+} + \pi_{+1} - 1.
\end{eqnarray}

The posterior distributions of $(\pi_{1+}, \pi_{+1})$ are the same as (\ref{eq::post}). However, the imputations of the missing potential outcomes are different from Section \ref{sec::ind}.
For $W_i = 1$, we impute
\begin{eqnarray*}
Y_i(0)|Y_i(1) = 1 &\sim & \Bern\left(   { \pi_{11} \over \pi_{1+}} =    \frac{\gamma \pi_{+1}}{ 1 - \pi_{+1} + \gamma \pi_{+1} }   \right),\\
Y_i(0)|Y_i(1) = 0 &\sim & \Bern\left(   {\pi_{01} \over 1-\pi_{1+}} =   { \pi_{+1} \over 1-\pi_{1+}} -  \frac{\gamma \pi_{1+} \pi_{+1}}{ (1 - \pi_{+1} + \gamma \pi_{+1})(1-\pi_{1+}) }\right).
\end{eqnarray*}
For $W_i = 0$, we impute
\begin{eqnarray*}
Y_i(1)|Y_i(0)=1 &\sim & \Bern\left(     {\pi_{11} \over \pi_{+1}  } =    \frac{\gamma \pi_{1+}}{ 1 - \pi_{+1} + \gamma \pi_{+1} }     \right),\\
Y_i(1)|Y_i(0)=0 &\sim & \Bern\left(    {\pi_{10} \over 1-\pi_{+1}} = \frac{   \pi_{1+}    }{ 1 - \pi_{+1} + \gamma \pi_{+1} }  \right).
\end{eqnarray*}

\begin{figure}[hbt]
\centering
\includegraphics[width = 0.8 \textwidth]{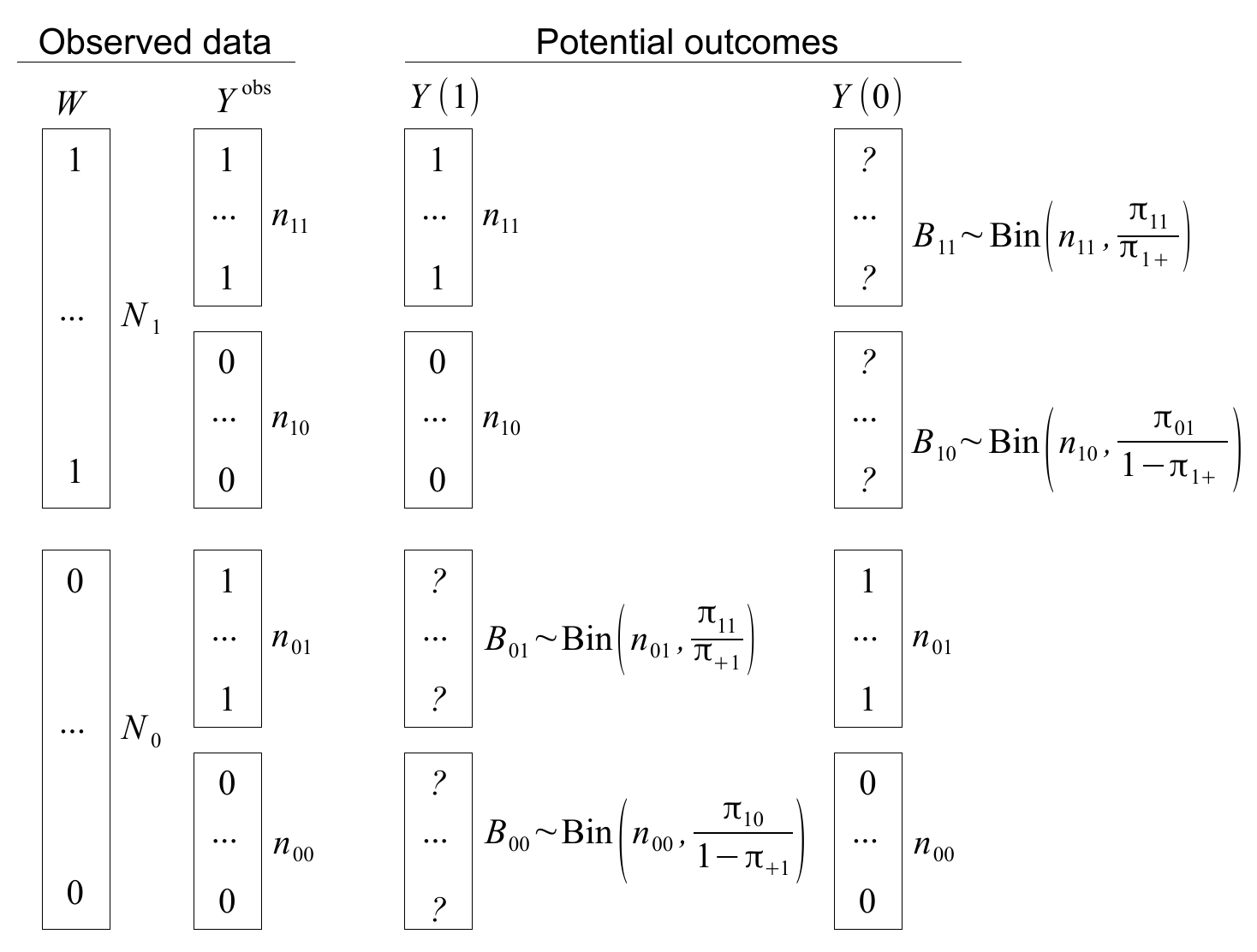}
\caption{Imputation of the Missing Potential Outcomes} \label{fg::impute}
\end{figure}

We illustrate the strategy for imputing missing potential outcomes in Figure \ref{fg::impute}, in which we have
$$
\tau| \bm{W}, \bm{Y}^{\text{obs}} , \pi_{1+}, \pi_{+1}\sim \frac{ n_{11} + B_{01} + B_{00} - B_{11} - B_{10} - n_{01} }{N},
$$
where $B_{11}\sim \Binomial(n_{11}, \pi_{11}/\pi_{1+})$, $B_{10}\sim \Binomial \left\{  n_{10}, \pi_{01}/(1 - \pi_{1+}) \right\} $,$ B_{01}\sim \Binomial(n_{01}, \pi_{11}/\pi_{+1})$, $B_{00}\sim \Binomial \left\{  n_{00}, \pi_{10}/(1-\pi_{+1}) \right\} $, and $\{ B_{11}, B_{10}, B_{01}, B_{00} \}$ are independent. Note that although the posterior distribution of $(\pi_{1+}, \pi_{+1})$ does not depend on the association parameter $\gamma$, the posterior distribution of $\tau$ does.
While there is no explicit form of the posterior distribution of $\tau$, we can approximate it via Monte Carlo.
We will apply the proposed sensitivity analysis
in Section \ref{sec::application}.

\section{\bf General causal measures}
\label{sec::general}

Up to now, we have considered the most commonly used causal estimand, the average causal effect (or $\CRD$). However, researchers and practitioners are also often
interested in the log of the causal risk ratio (relative risk)
\begin{eqnarray}
\log (\CRR) = \log (p_1) - \log (p_0) = \log\left(   \frac{ N_{11}+N_{10}   }{  N_{11}+N_{01}  }   \right) , \label{eq::log-CRR}
\end{eqnarray}
and
the log of the causal odds ratio
\begin{eqnarray}
\log (\COR )= \logit (p_1) - \logit (p_0) = \log\left( \frac{N_{11}+N_{10}}{ N_{01}+N_{00} }    \right)
- \log\left(   \frac{   N_{11}+N_{01} }{ N_{10}+N_{00} }   \right) , \label{eq::log-COR}
\end{eqnarray}
where logit$(x) = \log(x) - \log(1-x)$.
One attractive feature of $\CRD$ is that it is linear in the individual causal effects.
On the contrary, $\log (\CRR) $ and $\log(\COR)$ are finite population level causal estimands, which are not simple averages of individual causal effects.
The linearity of the average causal effect permitted Neyman (1923) to obtain an unbiased estimator with exact variance.
However, the elegant mathematics of Neyman (1923)'s randomization inference for $\CRD$ is not directly applicable to nonlinear causal measures.
We will fill in the gap by obtaining asymptotic randomization inference for $\log(\CRR)$ and $\log(\COR)$.

We can obtain estimators for the log of the causal risk ratio and odds ratio by substituting estimators of $p_1$ and $p_0$ in (\ref{eq::log-CRR}) and (\ref{eq::log-COR}), i.e.,
$\log (\widehat{\CRR} )= \log (\widehat{p}_1) - \log (\widehat{p}_0)$ and $\log (\widehat{\COR} )= \logit (\widehat{p}_1) - \logit (\widehat{p}_0)$.
As mentioned earlier, general nonlinear causal measures have not been studied carefully from the Neymanian perspective,
because the absence of linearity makes exact variance calculations intractable for such measures. Instead, we take an asymptotic perspective in this section. In the following subsections,
we will (i) propose asymptotic randomization inference for $\log (\widehat{\CRR} )$ and $\log (\widehat{\COR} )$; (ii) compare them with the results under traditional independent Binomial models; (iii) discuss Bayesian inference for the general causal measures.

\subsection{\bf Neymanian asymptotic randomization inference}

Unfortunately, the unbiasedness is not preserved by plugging $\widehat{p}_1$ and $\widehat{p}_0$ into the nonlinear functions (\ref{eq::log-CRR}) and (\ref{eq::log-COR}).
Furthermore, the plug-in estimators do not have finite means or variances,
since $\widehat{p}_1$ and $\widehat{p}_0$ can equal to $0$ or $1$ with positive probabilities.
In spite of these limitations, when $p_1$ and $p_0$ are both bounded away from $0$ and $1$, the estimators
$\log (\widehat{\CRR} )$ and $\log (\widehat{\COR} )$ have regular asymptotic distributions, summarized in the following two theorems.

\begin{theorem}
\label{thm::logCRR}
If $0<p_0,p_1<1$, as $N\rightarrow \infty$, $\log(\widehat{ \CRR} )$ is consistent for $\log(\CRR)$ and asymptotically Normal with asymptotic variance
\begin{eqnarray}
\frac{N_1 p_1 + N_0 p_0}{ N p_1 p_0 } \left(
\frac{S_1^2}{N_1 p_1}
+
\frac{S_0^2}{N_0 p_0}
- \frac{S_\tau^2}{N_1 p_1 + N_0 p_0}
\right).
\label{eq::asymptotic-variance-logCRR}
\end{eqnarray}
Assuming $S_\tau^2 = 0$ as in Neyman (1923), we can estimate the asymptotic variance by
\begin{eqnarray} \label{eq::var-logCRR-est}
\widehat{V}_{\CRR}
= \frac{n_{10}}{n_{11} N_1 } \frac{( n_{11} + n_{01})N_0 }{ n_{01} N} + \frac{n_{00}}{n_{01} N_0   }  \frac{(n_{11} + n_{01})  N_1 }{  n_{11} N } .
\end{eqnarray}
\end{theorem}

\begin{theorem}
\label{thm::logCOR}
If $0<p_0,p_1<1$, as $N\rightarrow \infty$, $\log(\widehat{ \COR} )$ is consistent for $\log(\COR)$ and asymptotically Normal with asymptotic variance
\begin{eqnarray}
 \frac{   N_1 p_1 (1 - p_1) + N_0 p_0 (1 - p_0)   }{Np_1(1-p_1)p_0 (1 - p_0)} \left\{
\frac{  S_1^2  }{N_1 p_1 (1 - p_1)}
+
\frac{ S_0^2  }{N_0 p_0 (1 - p_0)}
- \frac{ S_\tau^2}{  N_1 p_1 (1 - p_1) + N_0 p_0 (1 - p_0)     }
\right\}.
\label{eq::asymptotic-variance-logCOR}
\end{eqnarray}
Assuming $S_\tau^2 = 0$ as in Neyman (1923), we can estimate the asymptotic variance by
\begin{eqnarray}\label{eq::var-logCOR-est}
\widehat{V}_{\COR}
=  \frac{1}{n_{11}} + \frac{1}{n_{10}} + \frac{1}{n_{01}} + \frac{1}{n_{00}}.
   \end{eqnarray}
\end{theorem}

The variance formulae (\ref{eq::asymptotic-variance-logCRR}) and (\ref{eq::asymptotic-variance-logCOR}) for $ \log ( \widehat{\CRR}  )$ and $ \log ( \widehat{\COR}  ) $ are similar to the variance formula (\ref{eq::var}) for $\widehat{\tau}$, depending on the finite population variances of the potential outcomes $S_1^2$ and $ S_0^2$, and the unidentifiable finite population variance of the individual causal effect $S_\tau^2$.

Furthermore, borrowing the idea of bias-correction for ratio estimators (Cochran 1977), we can obtain bias-corrected estimators for $\log(\CRR)$ and $\log(\COR)$, which have lower order asymptotic biases than the na\"ive moment estimators.
Similar to Neyman's variance estimator for $\var(\widehat{\tau})$,
the variance estimators in (\ref{eq::var-logCRR-est}) and (\ref{eq::var-logCOR-est}) are conservative unless the constant causal effects assumption holds. Analogous to the result in (\ref{eq::improve-neyman-CRD}) for $\CRD$, using the lower bound for $S_\tau^2$ in Theorem \ref{thm::bound-tau}, we can improve the variance estimators (\ref{eq::var-logCRR-est}) and (\ref{eq::var-logCOR-est}) for the bias corrected estimators for $\log(\CRR)$ and $\log(\COR)$.
These bias-corrected point estimators and improved variance estimators improve the moment-based Neymanian inference asymptotically, and we call them improved Neymanian inference hereinafter.
We provide technical details about bias and variance reduction in Appendix A with proofs in the Supplementary Materials.

\subsection{\bf Independent Binomial models versus Neymanian inference}
\label{subsec::binomial}

In current clinical practice, the following independent Binomial models are widely used:
\begin{eqnarray}
\label{eq::Binomial}
n_{11}\sim \Binomial(N_1, p_1), \quad
n_{01}\sim \Binomial(N_0, p_0), \quad
n_{11}\ind n_{01}.
\end{eqnarray}
In the model above, $n_{11}$ and $n_{01}$ are assumed to be Binomial random variables. Such an assumption cannot, however, be justified by randomization using the potential outcomes model.

The maximum likelihood estimators for $p_1 - p_0, \log(p_1) - \log(p_0), \logit(p_1) - \logit(p_0)$ are the same as
$\widehat{\tau},   \log (\widehat{\CRR}), \log (\widehat{\COR})$, and their asymptotic variances (Woolf 1955; Rothman et al. 2008) can be estimated by
\begin{eqnarray}
\widehat{V}_{\CRD}^\Bin &=& {\widehat{p}_1(1 - \widehat{p}_1) \over N_1 }
+{  \widehat{p}_0 (1 - \widehat{p}_0) \over N_0}  , \label{eq::var-CRD-binomial} \\
\widehat{V}_{\CRR}^\Bin &=&
\frac{1}{n_{11}} - \frac{1}{N_1} + \frac{1}{n_{01}} - \frac{1}{N_0}
=
\frac{n_{10}}{n_{11} N_1} + \frac{n_{00}}{n_{01} N_0 } , \label{eq::var-logCRR-binomial} \\
\widehat{V}_{\COR}^\Bin &=&  \frac{1}{n_{11}} + \frac{1}{n_{10}} + \frac{1}{n_{01}} + \frac{1}{n_{00}}.\label{eq::var-logCOR-binomial}
\end{eqnarray}

Here, the superscript ``$\Bin$'' is for ``Binomial'' models.
For $\CRD$ and $\log ( \COR )$, the estimated variances under independent Binomial models are the same as Neymanian inference assuming constant causal effects.
However, this does not hold for $\log ( \CRR) $.
One sufficient condition for the equivalence of the variances from Neymanian inference and independent Binomial models is
$$
\frac{N_1}{N_0} = \frac{n_{11}}{n_{01}}, \text{ or equivalently, } \widehat{p}_1 = \widehat{p}_0,
$$
which essentially assumes the null hypothesis of zero average causal effect.

However, all the conclusions here are based on the constant causal effects assumption which may not be realistic in applications with binary outcomes.
Without assuming constant causal effects and by using the new sharp bound for $S_\tau^2$ in (\ref{eq::bound-s-tau}), we obtain different results from independent Binomial models, as shown in Appendix A. One surprising property of the log odds ratio is that the variance estimator under independent Binomial models (\ref{eq::var-logCOR-binomial}) is symmetric with respect to treatment and outcome, which coincides with the randomization-based variance estimator (\ref{eq::var-logCOR-est}) assuming $S_\tau^2=0$. However, the true variance of $\log(\widehat{\COR})$ over all possible randomizations, (\ref{eq::asymptotic-variance-logCOR}), and the improved variance estimator in Appendix A, (\ref{eq::var-logCOR-neyman}), do not have this symmetry.

\subsection{\bf Bayesian inference for general causal measures}

As shown above, Neymanian randomization inference for nonlinear measures of causal effects involves tedious algebra, and relies
on asymptotics under regularity conditions. In contrast, the Bayesian inference for $\log(\CRR)$ and $\log(\COR)$ is quite natural, once we impute all the missing potential outcomes based on their posterior predictive distributions.

For example, under the independent potential outcomes model, we have
\begin{eqnarray*}
\log (\CRR) |\bm{W}, \bm{Y}^{\text{obs}}, \pi_{1+}, \pi_{+1}& \sim&
\log \left(    \frac{ n_{11} + B_0 }{ n_{01} + B_1 } \right)  , \\
\log (\COR) |\bm{W}, \bm{Y}^{\text{obs}}, \pi_{1+}, \pi_{+1} &\sim&
\log\left(   \frac{n_{11} + B_0}{N - n_{11} -  B_0} \right)
- \log\left(     \frac{n_{01} + B_1}{N - n_{01} - B_1} \right) .
\end{eqnarray*}
Also, we can apply the Bayesian sensitivity analysis technique, similar to Section \ref{sec::BSA}, and obtain
\begin{eqnarray*}
\log (\CRR) |\bm{W}, \bm{Y}^{\text{obs}}, \pi_{1+}, \pi_{+1} &\sim&
\log \left(   \frac{n_{11} + B_{01} + B_{00} }{ n_{01} + B_{11} + B_{10} } \right)  ,\\
\log (\COR) |\bm{W}, \bm{Y}^{\text{obs}}, \pi_{1+}, \pi_{+1} &\sim&
\log\left(   \frac{n_{11} + B_{01} + B_{00} }{N - n_{11} -  B_{01}  - B_{00}  } \right)
- \log\left(     \frac{n_{01} + B_{11} - B_{10}  }{N - n_{01} - B_{11} - B_{10}} \right) .
\end{eqnarray*}
The posterior distributions of these causal measures can then be approximated by Monte Carlo.

\section{\bf Simulation Studies}

In order to compare the finite sample properties of Neyman's original method, the modified Neyman's method, and the Bayesian method assuming independent potential outcomes, we conduct two sets of simulation studies with independent and positively associated potential outcomes.
The first set listed as Cases 1--5 in Table \ref{tb::simulation} represent independent potential outcomes, while those listed as Cases 6--12 represent positively associated potential outcomes.
Table \ref{tb::simulation} also shows the marginal variances, correlations of the potential outcomes, and causal measures for each set of potential outcomes. To save space in the main text,
we present only the results for $\CRD$ and $\log(\COR)$. The results for $\log(\CRR)$ and the simulation studies for negatively associated potential outcomes are discussed in the online Supplementary Materials.

\begin{table}[ht]
\centering
\caption{``Science table'' for the simulation studies}\label{tb::simulation}
\begin{tabular}{|c|cccc|cccc|ccc|}
\hline
Case & $N_{11}$ & $N_{10}$ & $N_{01}$ & $N_{00}$ & $S_1^2$ & $S_0^2$ & $S_{10}$ &  $S_\tau^2$ & $\tau$ & $\log(\CRR)$ & $\log(\COR)$ \\
\hline
1&50&50&50&50&0.251 & 0.251 & 0.000 & 0.503 & 0.000 & 0.000 & 0.000 \\
2&30&70&30&70&0.251 & 0.211 & 0.000 & 0.462 & 0.200 & 0.511 & 0.847 \\
3&30&90&20&60&0.241 & 0.188 & 0.000 & 0.430 & 0.350 & 0.875 & 1.504 \\
4&80&20&80&20&0.251 & 0.161 & 0.000 & 0.412 & -0.300 & -0.470 & -1.386 \\
5&60&20&90&30&0.241 & 0.188 & 0.000 & 0.430 & -0.350 & -0.629 & -1.504 \\
\hline
6&60&40&40&60&0.251 & 0.251 & 0.050 & 0.402 & 0.000 & 0.000 & 0.000 \\
7&50&50&30&70&0.251 & 0.241 & 0.050 & 0.392 & 0.100 & 0.223 & 0.405 \\
8&50&70&30&50& 0.241 & 0.241 & 0.010 & 0.462 & 0.200 & 0.405 & 0.811 \\
9&40&110&10&40&0.188 & 0.188 & 0.013 & 0.352 & 0.500 & 1.099 & 2.197 \\
10&70&30&50&50&0.251 & 0.241 & 0.050 & 0.392 & -0.100 & -0.182 & -0.405 \\
11&50&30&70&50&0.241 & 0.241 & 0.010 & 0.462 & -0.200 & -0.405 & -0.811 \\
12&30&10&110&50& 0.161 & 0.211 & 0.010 & 0.352 & -0.500 & -1.253 & -2.234 \\
\hline
\end{tabular}
\end{table}

For given potential outcomes, we draw, repeatedly and independently, the treatment assignment vectors $5000$ times,
obtain the observed outcomes, and then apply three methods: Neymanian inference assuming constant treatment effects, improved Neymanian inference, and Bayesian inference assuming independent potential outcomes.
The improved Neymanian inference means using the improved variance estimator (\ref{eq::improve-neyman-CRD}) for $\CRD$, and bias-corrected estimators (\ref{eq::logCRR-bias}) and (\ref{eq::logCOR-bias}) and improved variance estimators (\ref{eq::var-logCRR-neyman}) and (\ref{eq::var-logCOR-neyman}) for nonlinear causal measures $\log(\CRR)$ and $\log(\COR)$.
Comparison of these methods are summarized in Figure \ref{fg::simulation}, with average biases, average lengths of the $95\%$ confidence/credible intervals, and the coverage probabilities.

First, the bias-corrected estimators for nonlinear causal measure $\log(\COR)$ do have smaller biases than the original Neymanian estimators and Bayes estimators in most cases.
Second, the confidence intervals from the modified Neyman's method are narrower than Neyman's original method, while still maintaining correct coverage properties. They indeed improve Neyman's original method.
Third,
the average widths of the Bayesian credible intervals are much narrower than the original and modified Neyman's method.
Moreover, when the potential outcomes are independent, the Bayesian credible intervals have correct frequentists' coverage property.  When the potential outcomes are positively associated and the average causal effect is small, the Bayesian credible intervals slightly under-cover the truth. The results in the Supplementary Materials show that when the potential outcomes are negatively associated, even the narrowest Bayesian credible intervals over-cover the true causal measures, and the Neymanian intervals and their modification are too ``conservative.''

As suggested by a reviewer, it is also interesting to investigate the frequency coverage property of the improved variance estimator (\ref{eq::improve-neyman-CRD}) for $\CRD$ under the sharp null. Table \ref{tb::sharp-null} compares the frequency properties of Neymanian variance estimator (\ref{eq::var-neyman}) and its improved version (\ref{eq::improve-neyman-CRD}), with moderate sample size $N=30$ and different choices of $N_{11}$ and $N_{00}$ such that $N_1 = N_0 = 15$. Except for the case with $N_{11}=N_{00}=15$, the improved variance estimators have shorter confidence intervals but preserve the same coverage rates. Under the sharp null, with sample size $N=30$, the sampling variance of $\widehat{\tau}$ is maximized at $N_{11}=N_{00}=15$, and in this case the adjusting term $| \widehat{\tau} | (1- | \widehat{\tau} | )/(N-1)$ has the wildest behavior. Therefore, in practice, if we have small sample sizes and observe that $\widehat{p}_1\approx \widehat{p}_0 \approx 0.5$, the improved variance estimator may hurt our inference. In all other situations, we suggest using the improved variance estimator.

\begin{table}[ht]
\center
\caption{Neymanian and its improved variance estimators for $\CRD$, (\ref{eq::var-neyman}) and (\ref{eq::improve-neyman-CRD}), under the sharp null hypothesis with $N=30$ and $N_{10} = N_{01}= 0$}\label{tb::sharp-null}
\begin{tabular}{|cc|cc|cc|}
\hline
$N_{11}$ &$N_{00}$ & Length & Coverage & Length$^c$ & Coverage$^c$  \\
 \multicolumn{1}{|c}{} & \multicolumn{1}{c|}{} & \multicolumn{2}{c|}{(Using Neyman's estimator)} & \multicolumn{2}{c|}{(Using corrected estimator)} \\ \hline
$20$&$10$&$0.686$&$0.951$&$0.644$&$0.951$\\
$25$&$5$&$0.542$&$0.959$&$0.491$&$0.959$\\
$15$&$15$&$0.728$&$0.971$&$0.683$&$0.858$\\
$12$&$18$&$0.713$&$0.942$&$0.672$&$0.942$\\
$8$&$22$&$0.633$&$0.96$&$0.601$&$0.96$\\
\hline
\end{tabular}
\end{table}

\begin{figure}[ht]
\centering
\subfigure[Independent Potential Outcomes: Cases 1 to 5 with the x-axis denoting the case numbers]{
\includegraphics[width=\textwidth]{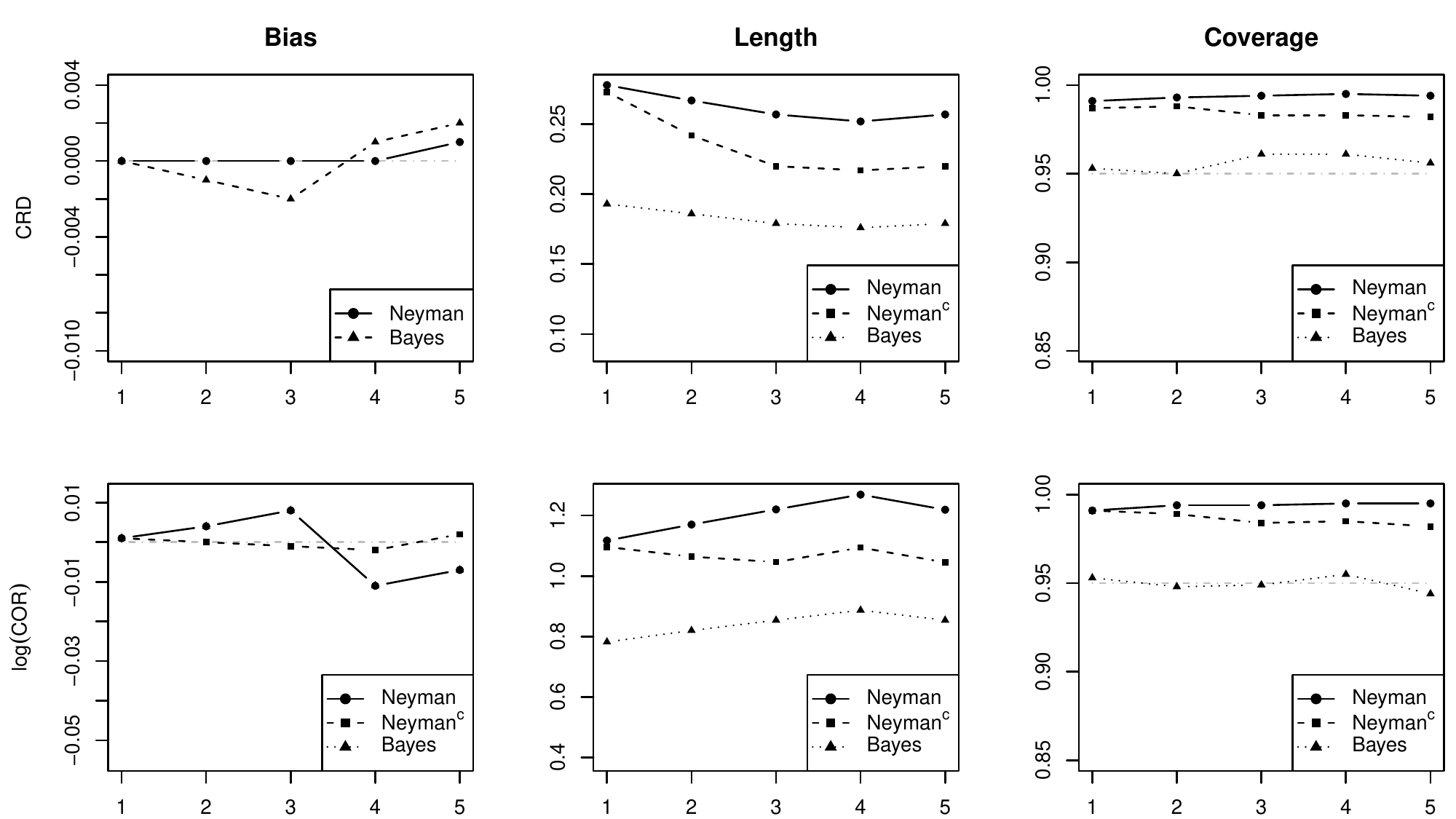}
}
\subfigure[Positively Associated Potential Outcomes: Cases 6 to 13 with the x-axis denoting the case number]{
\includegraphics[width=\textwidth]{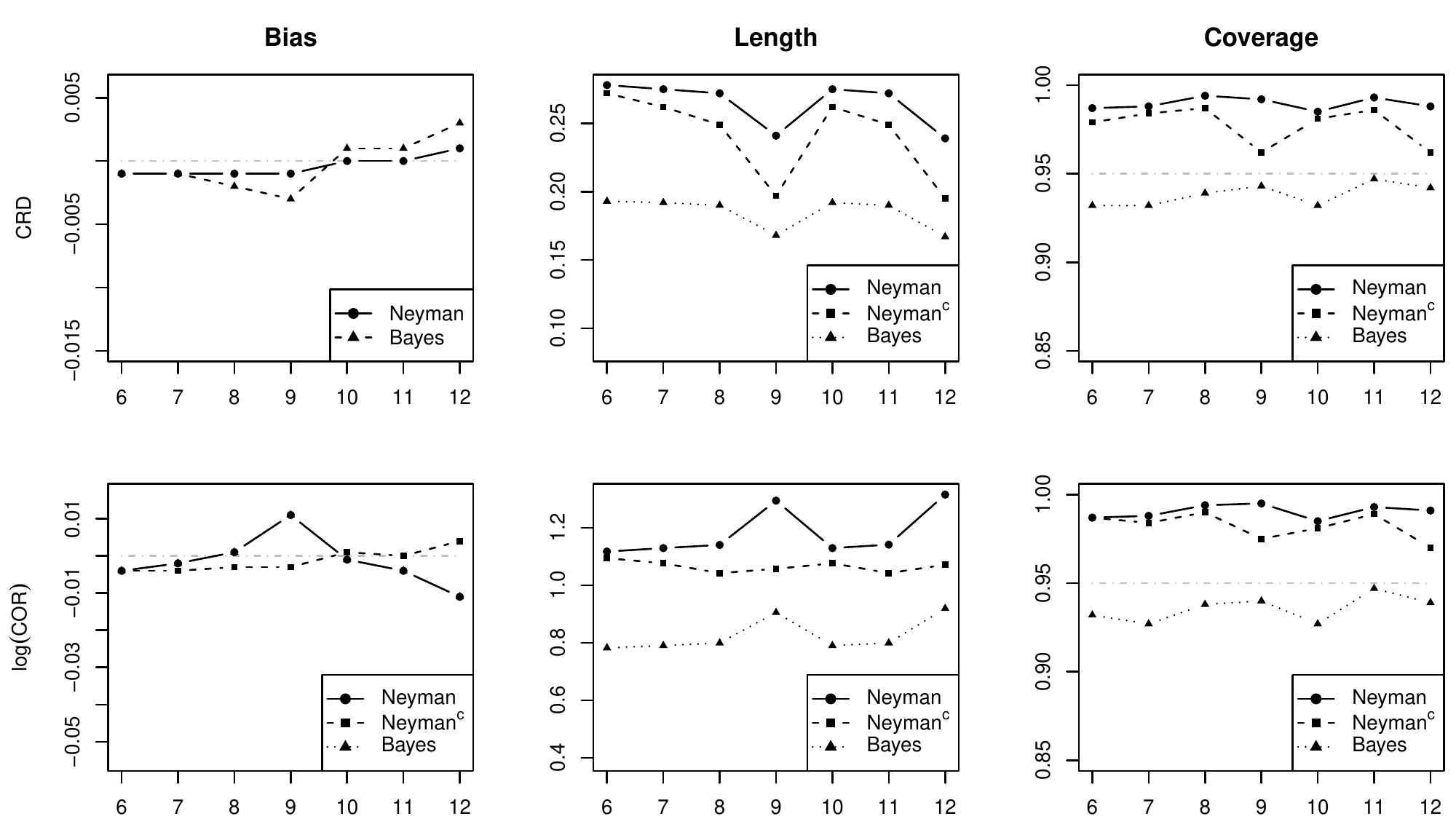}
}
\caption{Simulation Studies. Each subfigure is a $2\times 3$ matrix summarizing results for $2$ parameters and $3$ properties.
Note that ``Neyman'' and ``Bayes'' are indistinguishable for biases of $\log(\COR)$.}\label{fg::simulation}
\end{figure}

\section{\bf Application to a Randomized Controlled Trial}
\label{sec::application}

This example is taken from Bissler et al. (2013), where the authors compare the rate of adverse events in the treatment group versus the control group. The adverse event naspharyngitis occurred in $19$ out of $79$ subjects in the treatment group with everolimus, and it occurred in $12$ out of $39$ subjects in the control group. Therefore, the $2\times 2$ table representing the observed data has cell counts $(n_{11}, n_{10}, n_{01}, n_{00}) = (19,60,12,27).$
In Figure \ref{fg::point-application}, we show the results for three causal measures using Neymanian inference, modified Neymanian inference, and Bayesian posterior inference assuming independent potential outcomes.
The results match with those in our simulation studies in the sense that the bias-corrected estimators are slightly different from the original estimators, and the Bayes posterior credible intervals are much narrower than the confidence intervals from Neymanian inference.
However, in this particular example, all intervals cover zero.

Since the independence assumption between potential outcomes has a strong impact on the Bayesian inference for the finite population causal measures, we conduct a sensitivity analysis as proposed in Section \ref{sec::BSA} by varying $\log(\gamma)$ within $[-2, 4]$, and obtain Bayesian credible intervals for the causal measures at each $\log(\gamma)$.
Figure \ref{fg::sensitivity-application} shows the sensitivity analysis for $\CRD$ and $\log(\COR)$, with similar patterns $\log(\CRR)$ as shown in the Supplementary Materials.
Finally, the widths of the credible interval depend on $\log(\gamma)$; however, in the example, even the widest credible intervals are narrower than the ``conservative'' Neymanian confidence intervals.

\begin{figure}[ht]
\centering
\subfigure[
Inference for the Causal Measures.
We apply Neymanian, improved Neymanian (Neyman$^c$ above) and Bayesian approaches with
the segments representing the $95\%$ confidence/credible intervals and centers illustrating the point estimators.]{
\label{fg::point-application}
\begin{tabular}{ccc}
$\CRD$ & $\log(\CRR)$ & $\log(\COR)$\\
\includegraphics[width = 0.3\textwidth]{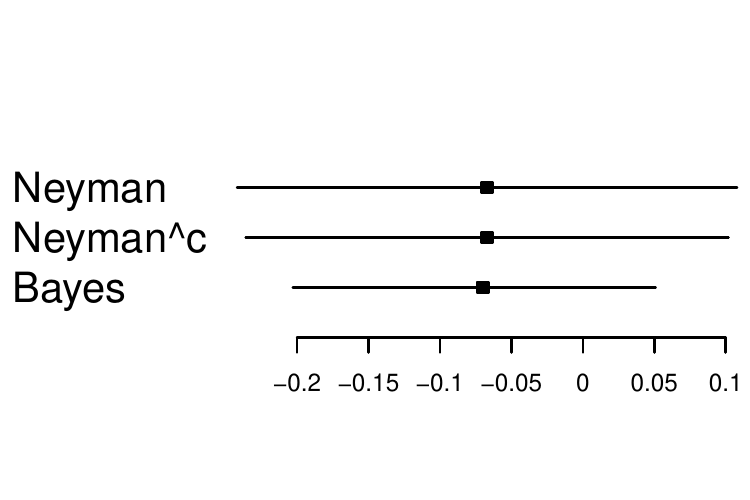} &
\includegraphics[width = 0.3\textwidth]{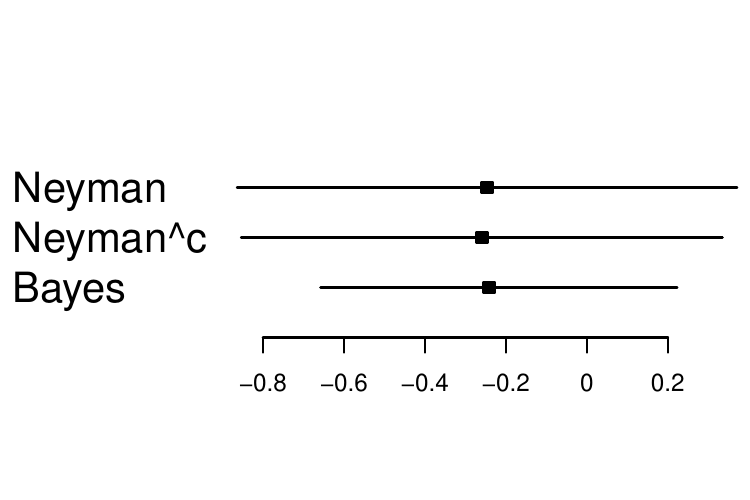}&
\includegraphics[width = 0.3\textwidth]{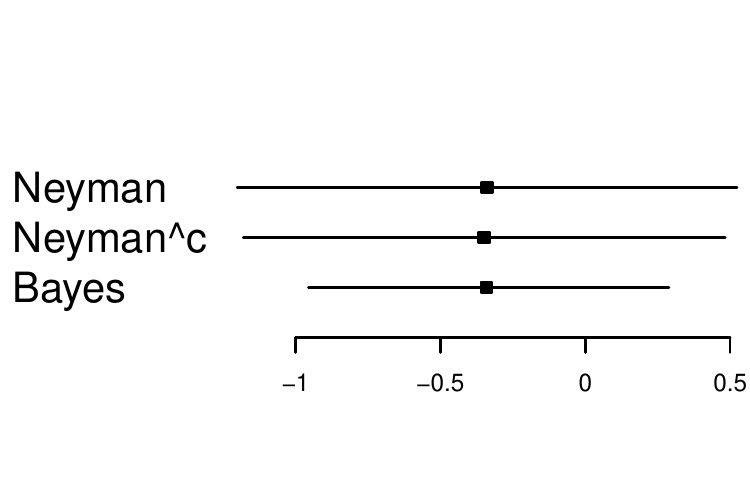}\\
\end{tabular}
}
\subfigure[Bayesian Sensitivity Analysis for $\CRD$ and $\log(\COR)$.
The intervals named ``independence'' are the $95\%$ posterior credible intervals under independence of the potential outcomes,  and the intervals named ``widest'' are the widest $95\%$ credible intervals over the ranges of the sensitivity parameters.]{\label{fg::sensitivity-application}
\includegraphics[width = \textwidth]{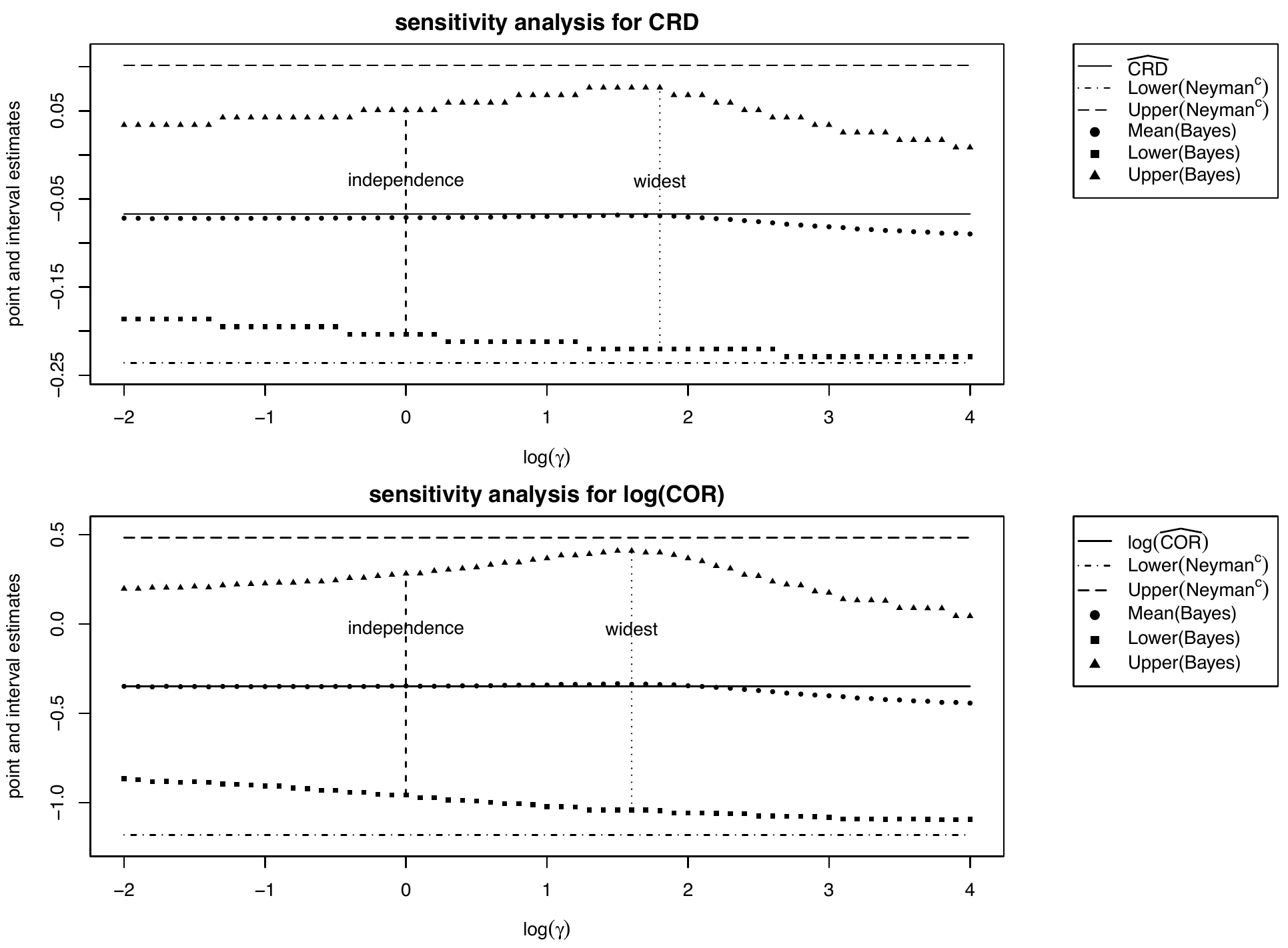}
}
\caption{A Randomized Experiments with Observed Data $(n_{11}, n_{10}, n_{01}, n_{00}) = (19,60,12,27)$.}
\end{figure}

\section{\bf Discussion}

In this paper, we have discussed causal inference of completely randomized treatment-control studies with binary outcomes under the potential outcomes model. We first made a connection between the Fisher randomization test (Fisher 1935a) and Fisher's exact test (Fisher 1935b) for binary outcomes, and proposed a procedure which uniformly dominates Neyman (1923)'s method. Although widely used in clinical practice, statistical inference for general nonlinear causal measures are based on the assumption of independent Binomial models, which is not justified by randomization. Based on randomization, our asymptotic analysis shows that the widely used variance estimators are either incorrect or inefficient, unless the null hypothesis of zero average causal effect is true.

Ding (2014) shows that the Neyman's test for zero average causal effect tends to be more powerful than Fisher's test for zero individual causal effect for many realistic cases including balanced designs. Our variance estimator (\ref{eq::improve-neyman-CRD}) further improves Neyman's test. Our new result is not contradictory to the classical result that Fisher's exact test is the uniformly most powerful unbiased test for equal probability of two independent Binomials (Lehmann and Romano, 2006). Both the Fisherian and Neymanian approaches are derived under the potential outcomes model, but the classical result for Fisher's exact test is derived under the independent Binomial models.

Traditionally, the variance formulae in (\ref{eq::var-logCOR-est}) and (\ref{eq::var-logCOR-binomial}) for $\log(\COR)$ have been used in both experimental and observational studies (including both prospective and retrospective observational studies).
Due to the symmetry of the variance formulae in (\ref{eq::var-logCOR-est}) and (\ref{eq::var-logCOR-binomial}) with respect to the treatment and outcome, researchers found that statistical inference of the log odds ratio measure is invariant to the sampling scheme (experimental study, prospective or retrospective observational studies), which was regarded as a celebrated and also mysterious feature of the log of the odds ratio.
As a pioneer in epidemiology and biostatistics, Cornfield (1959) said that
``there is a distinction seems undeniable, but its exact nature is elusive,'' when he was discussing experimental and observational studies.
However, randomized experiments are fundamentally different from observational studies, and especially different from retrospective observational studies.
Under the potential outcomes model with the potential outcomes treated as fixed quantities, the randomness of the observed outcomes comes only from the physical randomization in experiments.
Therefore, the treatment and the observed outcome are asymmetric unless the sharp null is true, and the variance in (\ref{eq::asymptotic-variance-logCOR}) for $\log(\COR)$ and its estimator in (\ref{eq::var-logCOR-neyman}) reflect the asymmetric nature explicitly.
In a recent comment on Cornfield (1959), Rubin (2012) suggested revealing the hidden nature of different studies using potential outcomes. Indeed, our results verify Rubin (2012)'s conjecture.

In order to reveal the importance of the correlation between potential outcomes and the intrinsic lack of additivity for binary outcomes, we focus our discussion on two by two tables from completely randomized experiments. The same idea can be also applied to observational studies as long as the ignorability assumption holds. We can either stratify on the observed covariates or propensity scores (Rosenbaum and Rubin 1983), and then within each strata the data can be approximately viewed as generated from randomized experiments (Rosenbaum 2002).
The findings of this paper can be generalized in many ways. For instance, we can discuss Neymanian randomization inference for full factorial or fractional factorial designs with binary outcomes, since the current discussion (Dasgupta et al. 2015) is restricted to continuous outcomes. It will also be interesting to discuss causal inference under the potential outcomes model for general outcomes (categorical data, counts, survival times, etc.). These topics are our ongoing or future research projects.

\appendix

\makeatletter   
 \renewcommand{\@seccntformat}[1]{APPENDIX~{\csname the#1\endcsname}.\hspace*{1em}}
 \makeatother

\section{\bf Bias and Variance Reduction for Nonlinear Causal Measures}

\noindent {\bfseries Result for $\log(\CRR)$.}
A bias-corrected estimator for $\log (\CRR)$ is
\begin{eqnarray}
\label{eq::logCRR-bias}
\log (\widehat{\CRR})^c =
\log (\widehat{\CRR}) +  \frac{  N_0}{ 2 \widehat{p}_1^2 N_1 N}   s_1^2 - \frac{N_1}{ 2 \widehat{p}_0^2 N_0 N }s_0^2 ,
\end{eqnarray}
with improved variance estimator
\begin{eqnarray}
\label{eq::var-logCRR-neyman}
\widehat{V}_{\CRR}^c  =
\widehat{V}_{\CRR}
- \frac{|\widehat{\tau}| (1 - |\widehat{\tau}|)  }{ \widehat{p}_1 \widehat{p}_0 (N-1) },
\end{eqnarray}
where $\widehat{V}_{\CRR}$ is defined in (\ref{eq::var-logCRR-est}).

\noindent {\bfseries Result for $\log(\COR)$.}
\label{thm::logCOR-bias-var}
A bias-corrected estimator for $\log (\COR)$ is
\begin{eqnarray}
\label{eq::logCOR-bias}
\log (\widehat{\COR})^c &=&
\log( \widehat{\COR} )+  \frac{1 - 2 \widehat{p}_1}{2 \widehat{p}_1^2 (1 - \widehat{p}_1)^2 }  \frac{N_0}{  N_1 N} s_1^2
- \frac{1 - 2 \widehat{p}_0}{2 \widehat{p}_0^2(1 - \widehat{p}_0)^2} \frac{N_1}{N_0 N} s_0^2 ,
\end{eqnarray}
with improved variance estimator
\begin{eqnarray}\label{eq::var-logCOR-neyman}
\widehat{V}_{\COR}^c=
\widehat{V}_{\COR}
- \frac{|\widehat{\tau}| (1 - |\widehat{\tau}|) }{  \widehat{p}_1(1 - \widehat{p}_1)\widehat{p}_0 (1 - \widehat{p}_0) (N-1)      },
\end{eqnarray}
where $\widehat{V}_{\COR}$ is defined in (\ref{eq::var-logCOR-est}).

\section*{\bf Supplementary Materials}
In the Supplementary Materials, Appendix B contains two useful lemmas and their proofs, and Appendices C and D provide proofs of all the theorems and results in Appendix A above.
Appendix E presents details about the simulation studies, especially the case with negatively associated potential outcomes and relatively small sample sizes.
Appendix F gives more details about the example in Section \ref{sec::application}.

\mbox{}\vspace*{1ex}
\mbox{}

%

\section*{\bf References}
\begin{description} \itemsep=-\parsep \itemindent=-1.2 cm

\item
Angrist, J. D., Imbens, G. W., and Rubin, D. B. (1996). Identification of causal effects using instrumental variables (with discussion). {\it Journal of the American Statistical Association}, {\bfseries 91}, 444--455.


\item
Berkson, J. (1978). Do the marginal totals of the $2\times 2$ table contain relevant information respecting the table proportions? (with discussion) {\it Journal of Statistical Planning and Inference}, {\bfseries 2}, 43--44.

\item
Bissler, J. J. et al. (2013). Everolimus for angiomyolipoma associated with tuberous sclerosis complex or sporadic lymphangioleiomyomatosis (EXIST-2): a multicentre, randomised, double-blind, placebo-controlled trial. {\it The Lancet}, {\bfseries 381}, 817--824.


\item
Chernoff, H. (2004). Information, for testing the equality of two probabilities, from the margins of the $2\times 2$ table. {\it Journal of Statistical Planning and Inference}, {\bfseries 121}, 209--214.


\item
Cochran, W. G. (1977). {\it Sampling Techniques}. John Wiley \& Sons: New York.

\item
Copas, J. B. (1973). Randomization models for the matched and unmatched $2\times 2$ tables. {\it Biometrika}, {\bfseries 60}, 467--476.


\item
Cornfield J. (1959). Principles of research. {\it American Journal of Mental Deficiency}, {\bfseries 64}, 240--252.


\item
Cox, D. R. (1958). {\it Planning of Experiments}. Wiley: New York.

%

\item Dasgupta, T., Pillai, N., and Rubin, D. B. (2015).
Causal inference from $2^k$
factorial designs using the potential
outcomes model. {\it Journal of the Royal Statistical Society, Series B (Statistical Methodology)}, in press, available at
\url{http://arxiv.org/abs/1211.2481}.


\item Ding, P. (2014). A paradox from randomization-based causal inference. Available at \url{http://arxiv.org/abs/1402.0142}.

\item Fisher, R. A. (1935a). {\it The Design of Experiments, 1st edn}. Oliver \& Boyd: Oxford, England.

\item
Fisher, R. A. (1935b). The logic of inductive inference. {\it Journal of the Royal Statistical Society}, {\bfseries 98}, 39--82.

\item
Gadbury, G. L. (2001). Randomization inference and bias of standard errors. {\it The American Statistician}, {\bfseries 55}, 310--313.

%
%
%
%

%
%



\item Kempthorne, O. (1952) {\it The Design and Analysis of Experiments}. Wiley: New York.

\item Lehmann, E. L. and Romano, J. P. (2006). {\it Testing Statistical Hypotheses}. Springer: New York.

\item Neyman, J. (1923). On the application of probability theory to agricultural experiments. Essay on principles (with discussion). Section 9 (translated). {\it Statistical Science}, {\bfseries  5}, 465--480.

%
%
%

%
%


\item
Robins, J. M. (1988).
Confidence intervals for causal parameters.
{\it Statistics in Medicine}, {\bfseries 7}, 773--785.

\item
Rosenbaum, P. R. and Rubin, D. B. (1983). The central role of the propensity score in observational studies for causal effects. {\it Biometrika}, {\bfseries 70}, 41--55.

\item
Rosenbaum, P. R. (2002).{\it  Observational Studies}. Springer: New York.

\item
Rothman, K. J., Greenland, S., and Lash, T. L. (2008). {\it Modern Epidemiology}. Lippincott Williams \& Wilkins: Philadelphia, PA.

\item Rubin, D. B. (1974). Estimating causal effects of treatments in randomized and nonrandomized studies. {\it Journal of Educational Psychology}, {\bfseries 66}, 688--701.

%

\item Rubin, D. B. (1977). Assignment to treatment group on the basis of a covariate. {\it Journal of Educational Statistics}, {\bfseries 2}, 1--26.

\item Rubin, D. B. (1978). Bayesian inference for causal effects: The role of randomization. {\it The Annals of Statistics}, {\bfseries 6}, 34--58.

\item Rubin, D. B. (1980). Comment on ``Randomization analysis of experimental data: the Fisher randomization test''
by D. Basu. {\it Journal of the American Statistical Association,} {\bfseries 75,} 591--593.

\item Rubin, D. B. (1984). Bayesianly justifiable and relevant frequency calculations for the applied statistician. {\it The Annals of Statistics}, {\bfseries 12}, 1151--1172.

\item Rubin, D. B. (1990). Neyman (1923) and causal inference in experiments and observational studies. {\it Statistical Science}, {\bfseries 5}, 472--480.

%

\item Rubin, D. B. (2005). Causal inference using potential outcomes: design, modeling, and decisions. {\it Journal of the American Statistical Association}, {\bfseries 100}, 322--331.

%

\item Rubin, D. B. (2010). Reflections stimulated by the comments of Shadish (2010) and West and Thoemmes (2010). {\it Psychological Methods}, {\bfseries 15}, 38--40.

\item
Rubin, D. B. (2012). Potential updates to Cornfield's 1959 ``Principles of Research''. {\it Statistics in Medicine}, {\bfseries 31}, 2778--2779.

%
%
%
%


\item
Woolf, B. (1955). On estimating the relation between blood group and disease. {\it Annals of Human Genetics}, {\bfseries 19}, 251--253.

\item Yates, F. (1984). Tests of significance for $2\times 2$ contingency tables (with discussion). {\it Journal of the Royal Statistical Society, Series A}, {\bfseries 147}, 426--463.

\end{description}

\begin{center}
\bf \Large  Supplementary Materials for \\
``A Potential Tale of Two by Two Tables from Completely Randomized Experiments''
\end{center}

 \appendix

\setcounter{equation}{0}
\setcounter{section}{1}

\renewcommand {\thefigure} {A.\arabic{figure}}
\renewcommand {\thetable} {A.\arabic{table}}
\renewcommand {\theequation} {A.\arabic{equation}}
\renewcommand {\thelemma} {A.\arabic{lemma}}
\makeatletter   
 \renewcommand{\@seccntformat}[1]{APPENDIX~{\csname the#1\endcsname}.\hspace*{1em}}
 \makeatother

In the Supplementary Materials, Appendix B contains two useful lemmas and their proofs, and Appendix C and D provides proofs for all the theorems and the results in Appendix A above.
Appendix E presents details about the simulation studies, especially the case with negatively associated potential outcomes.
Appendix F gives more details about the example in Section 7.

\section{\bf Lemmas and Their Proofs}

\begin{lemma}
\label{lemma::randomization}
The completely randomized treatment assignment $\bm{W}$ satisfies $E(W_i) =N_1/N$, $\var(W_i) = N_1 N_0 /N^2$, and $\cov(W_i, W_j) = - N_1 N_0/ \{N^2 (N-1)\}$. If $(c_1, \cdots, c_N)$ and $(d_1, \cdots, d_N)$ are constants with $\overline{c} = \sumN c_i/N$ and $\overline{d} = \sumN d_i/N$, we have
\begin{eqnarray*}
E\left(  \sumN W_i c_i  \right) = N_1 \overline{c}, \quad 
\cov\left(  \sumN W_i c_i, \sumN W_i d_i \right) = \frac{N_1 N_0}{N (N-1)} \sumN (c_i - \overline{c}) (d_i - \overline{d}).
\end{eqnarray*}
\end{lemma}

\begin{proof}
[Proof of Lemma \ref{lemma::randomization}]
The observed outcomes in the treatment and control can be viewed as two sets of simple random samples from the finite population of $\{ Y_i(1):i=1,\cdots, N\} $ and $\{ Y_i(0): i=1,\cdots, N\}$, respectively. Therefore, the conclusion follows from classic survey sampling textbooks such as Cochran (1977).
\end{proof}

\begin{lemma}
\label{lemma:vcov}
The estimators, $\widehat{p}_1$ and $\widehat{p}_0$, are unbiased for $p_1$ and $p_0$, with variances and covariance:
\begin{eqnarray*}
\var(\widehat{p}_1)  = \frac{N_0}{  N_1 N} S_1^2, \quad
\var(\widehat{p}_0) =  \frac{N_1}{  N_0 N} S_0^2 ,\quad 
\cov(\widehat{p}_1 , \widehat{p}_0)   =   -\frac{1}{N} S_{10} =      -  \frac{1}{2N}  (S_1^2 + S_0^2 - S_\tau^2).
\end{eqnarray*}
\end{lemma}

\begin{proof}
[Proof of Lemma \ref{lemma:vcov}] 
The unbiasedness and variances of $\widehat{p}_1$ and $\widehat{p}_0$ follow directly from Lemma \ref{lemma::randomization}.
The covariance between $\widehat{p}_1$ and $\widehat{p}_0$ is
\begin{eqnarray*}
\cov(\widehat{p}_1 , \widehat{p}_0) =   - \frac{1}{N_1 N_0} \frac{N_1 N_0 }{ N } S_{10} = -\frac{1}{N} S_{10}.
\end{eqnarray*}
Summing from $i=1$ to $N$ over the following decomposition
\begin{eqnarray*}
 2\{ Y_i(1) -  p_1  \} \{  Y_i(0) - p_0 \} = \{   Y_i(1) - p_1 \}^2 +  \{ Y_i(0) - p_0 \} ^2 - (\tau_i - \tau)^2,
\end{eqnarray*}
we have
$
2S_{10} = S_1^2  + S_0^2 - S_\tau^2,
$
and therefore the covariance can also be expressed as
\begin{eqnarray*}
\cov(\widehat{p}_1 , \widehat{p}_0) = -  \frac{1}{2N}  (S_1^2 + S_0^2 - S_\tau^2).
\end{eqnarray*}
\end{proof}

\section{\bf Proofs of the Theorems}

\begin{proof}
[Proof of Theorem 1]
Define the
proportions $p_{jk} = N_{jk}/N$.
We first rewrite $S_\tau^2/N$ as
\begin{eqnarray*}
 \frac{S_{\tau}^2}{N} &=& \frac{1}{N(N-1)} \sumN (\tau_i - \tau)^2\\
& =&  \frac{1}{N(N-1)} \left\{ (N_{10}+N_{01}) - \frac{(N_{10} - N_{01})^2}{N}   \right\}\\
& =&   \frac{1}{N-1}(    \tau + 2p_{01}  - \tau^2  ).
\end{eqnarray*}
In order to find the lower bound of $S_\tau^2/N$, we only need to find the lower bound for $p_{01}$.
This reduces to the following linear programming problem:
\begin{eqnarray*}
\left\{
\begin{array}{ccl}
&\min\limits_{p_{11}, p_{10}, p_{01}, p_{00}} & p_{01} \\
&s.t.& p_{11} + p_{10} = p_1, \\
&&p_{11} + p_{01} = p_0, \\
&&p_{11}+p_{10}+p_{01}+p_{00} = 1,\\
&&   p_{ij} \geq 0, i, j = 0,1.
\end{array}\right.
\end{eqnarray*}
Since $p_{01} = p_0 - p_1 + p_{10} \geq -\tau$ and $p_{01} \geq 0$, the lower bound of $p_{01}$ is
$$
p_{01} \geq \max(-\tau, 0),
$$
and therefore, the lower bound of $S_\tau^2/N$ is
$$
{ S_\tau^2 \over N} \geq \frac{1}{N-1}\{     \tau + 2\max(-\tau, 0)  - \tau^2  \} = \frac{1}{N-1}\{     \max(-\tau, \tau)  - \tau^2  \}  = \frac{|\tau|(1 - |\tau|)}{N-1}.
$$
From the derivation above, the bound is sharp, and is attained if and only if $p_{10} = 0$ or $p_{01} = 0$. Or, equivalently, $S_\tau^2$ attains its minimum at either of the two vertices within the feasible region of the linear programming problem above: $(p_{11}, p_{10}, p_{01}, p_{00}) = (p_1, 0,  - \tau, 1-p_0)$ if $\tau\leq 0$, and $(p_{11}, p_{10}, p_{01}, p_{00}) = (p_0, \tau,0,1-p_1)$ if $\tau\geq 0$.
\end{proof}

\begin{proof}
[Proof of Theorem 2]
Define $N_w' = N_w + \alpha_w + \beta_w$ and $\widehat{p}_w' = (n_{w1} + \alpha_w) / N_w'$ as the sample sizes and proportions adjusted by the pseudo counts of the prior distributions.
The posterior means of $\pi_{1+}$ and $\pi_{+1}$ are
$$
E( \pi_{1+}\mid \bm{W}, \bm{Y}^{\text{obs}})  =  \widehat{p}_1',
\quad
\text{and}
\quad
E( \pi_{+1}\mid \bm{W}, \bm{Y}^{\text{obs}})  =  \widehat{p}_0'.
$$
The posterior variances of $\pi_{1+}$ and $\pi_{+1}$ are
$$
\var( \pi_{1+} \mid \bm{W}, \bm{Y}^{\text{obs}})  =  \frac{\widehat{p}_1'  (1 - \widehat{p}_1')}{N_1' + 1},
\quad
\text{and}
\quad
\var( \pi_{+1} \mid \bm{W}, \bm{Y}^{\text{obs}})  =  \frac{\widehat{p}_0'  (1 - \widehat{p}_0')}{N_0' + 1}.
$$
Immediately, we have
\begin{eqnarray*}
E\{  \pi_{1+} (1 -  \pi_{1+})\mid \bm{W}, \bm{Y}^{\text{obs}} \}   &=& \widehat{p}_1' (1 - \widehat{p}_1') - \frac{\widehat{p}_1' (1 - \widehat{p}_1')}{N_1' + 1}
=  \frac{N_1'}{N_1' + 1} \widehat{p}_1'(1 - \widehat{p}_1'),\\
E\{  \pi_{+1} (1 -  \pi_{+1})\mid \bm{W}, \bm{Y}^{\text{obs}} \}   &=& \widehat{p}_0' (1 - \widehat{p}_0') - \frac{\widehat{p}_0' (1 - \widehat{p}_0')}{N_0' + 1}
=  \frac{N_0'}{N_0' + 1} \widehat{p}_0'(1 - \widehat{p}_0') .
\end{eqnarray*}
Applying the laws of conditional expectation and variance to (11) in the main text, we obtain the posterior mean
\begin{eqnarray*}
&&E(\tau \mid \bm{W}, \bm{Y}^{\text{obs}}) \\
&=& E\left\{  E(\tau\mid \bm{W}, \bm{Y}^{\text{obs}}, \pi_{1+}, \pi_{+1})   \right\}  \\
&=& E\left(      \frac{n_{11} + N_0 \pi_{1+} - n_{01} - N_1 \pi_{+1} }{N}   \mid   \bm{W}, \bm{Y}^{\text{obs}} \right)\\
&=& \frac{  n_{11} + N_0 \widehat{p}_1' - n_{01} - N_1 \widehat{p}_0'   }{  N  }\\
&=& \frac{N_1' + N_0}{  N} \widehat{p}_1' - \frac{N_0' + N_1}{N} \widehat{p}_0' - \frac{\alpha_1 - \alpha_0}{N},
\end{eqnarray*}
and posterior variance
\begin{eqnarray*}
&&\var(\tau \mid \bm{W}, \bm{Y}^{\text{obs}}) \\
&=&   E\left\{  \var(\tau\mid \bm{W}, \bm{Y}^{\text{obs}}, \pi_{1+}, \pi_{+1})  \right\}  
+ \var\left\{ E(\tau\mid \bm{W}, \bm{Y}^{\text{obs}}, \pi_{1+}, \pi_{+1})   \right\} \\
&=&  E\left\{   \frac{N_0}{N^2} \pi_{1+} (1 - \pi_{1+})  + \frac{N_1}{N^2} \pi_{+1} (1 - \pi_{+1})    \mid \bm{W}, \bm{Y}^{\text{obs}}   \right\}   
+ \var\left\{     \frac{N_0}{N} \pi_{1+} - \frac{N_1}{N} \pi_{+1}      \mid \bm{W}, \bm{Y}^{\text{obs}}   \right\} \\
& = & \frac{N_0 N_1'}{ N^2 (N_1' + 1) } \widehat{p}_1'(1 - \widehat{p}_1')
+ \frac{N_1 N_0'}{N^2 (N_0' + 1)} \widehat{p}_0' (1 - \widehat{p}_0')
+ \frac{N_0^2}{N^2 (N_1' + 1)}  \widehat{p}_1' (1 - \widehat{p}_1')
+ \frac{N_1^2}{N^2 (N_0' + 1)} \widehat{p}_0' (1 - \widehat{p}_0')  \\
&=& \frac{N_0(N_1'+N_0)}{N^2}  \frac{\widehat{p}_1' (1 - \widehat{p}_1')}{N_1' + 1} 
+ \frac{N_1(N_1 + N_0') }{N^2} \frac{\widehat{p}_0' (1 - \widehat{p}_0')}{N_ 0' +1 } . 
\end{eqnarray*}
When we have large sample size, the prior pseudo counts are overwhelmed by the observed counts $n_{jk}$'s, and the posterior mean and variance of $\tau$ can be approximately by
\begin{eqnarray*}
E(\tau \mid \bm{W}, \bm{Y}^{\text{obs}})&\approx&\widehat{\tau},\\
\var(\tau \mid \bm{W}, \bm{Y}^{\text{obs}})&\approx&  \frac{N_0}{N}  \frac{\widehat{p}_1 (1 - \widehat{p}_1)}{N_1 - 1} + \frac{N_1 }{N} \frac{\widehat{p}_0 (1 - \widehat{p}_0)}{N_ 0 -1 } .  
\end{eqnarray*}
\end{proof}

\begin{proof}
[Proof of Theorem 3]
Applying Taylor expansion, we have
\begin{eqnarray*}
 \log (\widehat{\CRR} )- \log (\CRR)  = \frac{1}{p_1} (\widehat{p}_1 - p_1) 
 - \frac{1}{p_0} (\widehat{p}_0 - p_0) 
 +  o_p\left(  \frac{1}{N^{1/2}} \right).
\end{eqnarray*}
According to Lemma \ref{lemma:vcov}, the asymptotic variance of $  \log( \widehat{\CRR} )$ is
\begin{eqnarray*}
&&\frac{1}{p_1^2} \var( \widehat{p}_1) + \frac{1}{p_0^2} \var(\widehat{p}_0)
 - \frac{2}{p_1 p_0} \cov(\widehat{p}_1, \widehat{p}_0)  \\
 &=& \frac{1}{  p_1^2}   \frac{N_0}{  N_1 N} S_1^2 + \frac{1}{p_0^2}   \frac{N_1}{  N_0 N} S_0^2
 +  \frac{1}{p_1 p_0 N }   (S_1^2 + S_0^2 - S_\tau^2) \\
 &=& \frac{  N_1 p_1 + N_0 p_0}{   p_1^2 p_0 N_1 N } S_1^2
 + \frac{N_1 p_1 + N_0 p_0}{  p_1 p_0^2 N_0 N } S_0^2
 - \frac{1}{p_1 p_0 N} S_\tau^2 \\
 &=& \frac{N_1 p_1 + N_0 p_0  }{p_1 p_0 N} \left(
\frac{S_1^2}{N_1 p_1}
+
\frac{S_0^2 }{N_0 p_0}
- \frac{ S_\tau^2}{N_1 p_1 + N_0 p_0 }
\right).
\end{eqnarray*}
Assume $S_\tau^2 = 0$, and we can estimate the asymptotic variance by
\begin{eqnarray}
\widehat{V}_{\CRR} &=&
\frac{  N_1 \widehat{p}_1 + N_0 \widehat{p}_0}{   \widehat{p}_1^2 \widehat{p}_0 N_1 N } s_1^2
 + \frac{N_1 \widehat{p}_1 + N_0 \widehat{p}_0}{  \widehat{p}_1 \widehat{p}_0^2 N_0 N } s_0^2 \nonumber   \\
 &=& \frac{  N_1 \widehat{p}_1 + N_0 \widehat{p}_0}{   \widehat{p}_1^2 \widehat{p}_0 N_1 N } \frac{N_1}{N_1 - 1} \widehat{p}_1 (1 - \widehat{p}_1)
 + \frac{N_1 \widehat{p}_1 + N_0 \widehat{p}_0}{  \widehat{p}_1 \widehat{p}_0^2 N_0 N } \frac{N_0}{N_0 - 1}  \widehat{p}_0 (1 - \widehat{p}_0) \nonumber   \\
 &=& \frac{  (  N_1 \widehat{p}_1 + N_0 \widehat{p}_0)(1 - \widehat{p}_1) }{   \widehat{p}_1 \widehat{p}_0 (N_1 - 1) N }
 + \frac{  (N_1 \widehat{p}_1 + N_0 \widehat{p}_0 ) (1 - \widehat{p}_0)  }{  \widehat{p}_1 \widehat{p}_0 (N_0 - 1) N } \nonumber   \\
 &=& \frac{n_{10}}{n_{11} (N_1 - 1)} \frac{( n_{11} + n_{01})N_0 }{ n_{01} N} + \frac{n_{00}}{n_{01} (N_0 - 1)  }  \frac{(n_{11} + n_{01})  N_1 }{  n_{11} N }.
\end{eqnarray}
Ignoring the difference between $N_w$ and $(N_w-1)$ ($w=0,1$) in asymptotic analysis, we obtain the formula in Theorem 3.
\end{proof}

\begin{proof}
[Proof of Theorem 4]
Applying Taylor expansion, we have
\begin{eqnarray*}
 \log( \widehat{\COR}) - \log (\COR)
 = \frac{1}{p_1(1 - p_1)} (\widehat{p}_1 - p_1) 
 - \frac{1}{p_0 (1 - p_0)} (\widehat{p}_0 - p_0) 
 + o_p\left(  \frac{1}{N^{1/2}} \right).
\end{eqnarray*}
According to Lemma \ref{lemma:vcov}, the asymptotic variance of $  \log( \widehat{\COR} )$ is
\begin{eqnarray*}
&&  \frac{1}{p_1^2(1 - p_1)^2} \var( \widehat{p}_1)
+  \frac{1}{p_0^2 (1 - p_0)^2} \var(\widehat{p}_0)
 - \frac{2}{p_1 (1 - p_1) p_0  (1 - p_0) } \cov(\widehat{p}_1, \widehat{p}_0)   \\
 &=&  \frac{1}{p_1^2(1 - p_1)^2}  \frac{N_0}{  N_1 N} S_1^2
 + \frac{1}{p_0^2 (1 - p_0)^2}  \frac{N_1}{  N_0 N} S_0^2
 + \frac{1}{p_1 (1 - p_1) p_0  (1 - p_0) N }   (S_1^2 + S_0^2 - S_\tau^2) \\
 &=& \frac{   N_1 p_1 (1 - p_1) + N_0 p_0 ( 1 - p_0)  }{   p_1^2 (1 - p_1)^2 p_0 ( 1 - p_0) N N_1  }  S_1^2
 + \frac{N_1 p_1 (1 - p_1) + N_0 p_0 ( 1 - p_0) }{   p_1 (1 - p_1) p_0^2 (1 - p_0)^2 N N_0 } S_0^2
 - \frac{1}{  N p_1 (1 - p_1) p_0 (1 - p_0)  } S^2_\tau \\
 &=& \frac{N_1 p_1 (1 - p_1) + N_0 p_0 (1 - p_0)  }{Np_1(1-p_1)p_0 (1 - p_0)} \left\{
\frac{ S_1^2    }{N_1 p_1 (1 - p_1)}   +
\frac{S_0^2}{N_0 p_0 (1 - p_0)}
- \frac{ S_\tau^2}{N_1 p_1 (1 - p_1) + N_0 p_0 (1 - p_0)}
\right\}.
\end{eqnarray*}
The Neyman-type ``conservative'' variance estimator for the asymptotic variance is
\begin{eqnarray}
\widehat{V}_{\COR}&=&
\frac{   N_1 \widehat{p}_1 (1 -  \widehat{p}_1) + N_0  \widehat{p}_0 ( 1 -  \widehat{p}_0)  }{    \widehat{p}_1^2 (1 - \widehat{ p}_1)^2  \widehat{p}_0 ( 1 -  \widehat{p}_0) N N_1 } s_1^2
 + \frac{N_1  \widehat{p}_1 (1 - \widehat{ p}_1) + N_0  \widehat{p}_0 ( 1 -  \widehat{p}_0) }{   \widehat{ p}_1 (1 -  \widehat{p}_1)  \widehat{p}_0^2 (1 -  \widehat{p}_0)^2 N N_0 } s_0^2 \nonumber \\
 &=&
\frac{   N_1 \widehat{p}_1 (1 -  \widehat{p}_1) + N_0  \widehat{p}_0 ( 1 -  \widehat{p}_0)  }{    \widehat{p}_1 (1 - \widehat{ p}_1)  \widehat{p}_0 ( 1 -  \widehat{p}_0) N (N_1 - 1)  }
 + \frac{N_1  \widehat{p}_1 (1 - \widehat{ p}_1) + N_0  \widehat{p}_0 ( 1 -  \widehat{p}_0) }{   \widehat{ p}_1 (1 -  \widehat{p}_1)  \widehat{p}_0 (1 -  \widehat{p}_0) N (N_0 - 1) } \nonumber  \\
&\approx &\frac{   N_1 \widehat{p}_1 (1 -  \widehat{p}_1) + N_0  \widehat{p}_0 ( 1 -  \widehat{p}_0)  }{    \widehat{p}_1 (1 - \widehat{ p}_1)  \widehat{p}_0 ( 1 -  \widehat{p}_0)  } \left(  \frac{1}{NN_1} + \frac{1}{N N_0}  \right)\nonumber \\
&=& \frac{  (n_{01} + n_{00} ) n_{11} n_{10}   + (n_{11} + n_{10} )  n_{01} n_{00} }{   n_{11} n_{10} n_{01} n_{00}  }\nonumber  \\
&=& \frac{1}{n_{11}} + \frac{1}{n_{10}} + \frac{1}{n_{01}} + \frac{1}{n_{00}},
   \end{eqnarray}
where the approximation is due to the difference between $N_w-1$ and $N_w$ for $w=0,1.$
\end{proof}

\section{\bf Proofs for the Results in Appendix A about Bias and Variance Reduction for Nonlinear Causal Measures}

\begin{proof}
[Proof of the Result for $\log(\CRR)$]
Applying Taylor expansion, we have
\begin{eqnarray*}
 \log (\widehat{\CRR} )- \log (\CRR)  = \frac{1}{p_1 } (\widehat{p}_1 - p_1)  - \frac{1}{2p_1^2} (\widehat{p}_1 - p_1)^2
 - \frac{1}{p_0 } (\widehat{p}_0 - p_0) + \frac{1}{2p_0^2} (\widehat{p}_0 - p_0)^2 +  o_p\left(  \frac{1}{N}   \right).
\end{eqnarray*}
Therefore, the asymptotic bias of $  \log (\widehat{\CRR} )$ is
\begin{eqnarray*}
 - \frac{1}{2p_1^2}  \var(\widehat{p}_1) + \frac{1}{2p_0^2} \var(\widehat{p}_0)
 =  - \frac{  N_0}{ 2 p_1^2 N_1 N}   S_1^2 + \frac{N_1}{ 2 p_0^2 N_0 N }S_0^2,
\end{eqnarray*}
and the bias-corrected estimator for $\log (\CRR)$ in Appendix A can be obtained by subtracting the estimated asymptotic bias from $  \log (\widehat{\CRR} )$.
\end{proof}

\begin{proof}
[Proof of the Result for $\log(\COR)$]
Applying Taylor expansion, we have
\begin{eqnarray*}
&& \log( \widehat{\COR}) - \log (\COR)  \\
 &=& \frac{1}{p_1(1 - p_1)} (\widehat{p}_1 - p_1)  - \frac{1 - 2p_1}{2p_1^2 (1 - p_1)^2 } (\widehat{p}_1 - p_1)^2 \\
&& - \frac{1}{p_0 (1 - p_0)} (\widehat{p}_0 - p_0) + \frac{1 - 2p_0}{2p_0^2(1 - p_0)^2} (\widehat{p}_0 - p_0)^2 +   o_p\left(  \frac{1}{N}   \right) .
\end{eqnarray*}
Therefore, the asymptotic bias of $  \log( \widehat{\COR} )$ is
\begin{eqnarray*}
- \frac{1 - 2p_1}{2p_1^2 (1 - p_1)^2 } \var(\widehat{p}_1)
+ \frac{1 - 2p_0}{2p_0^2(1 - p_0)^2} \var(\widehat{p}_0)  
=  - \frac{1 - 2p_1}{2p_1^2 (1 - p_1)^2 }  \frac{N_0}{  N_1 N} S_1^2
+ \frac{1 - 2p_0}{2p_0^2(1 - p_0)^2} \frac{N_1}{N_0 N} S_0^2 ,
\end{eqnarray*}
and the bias-corrected estimator for $\log (\COR)$ in Appendix A can be obtained by subtracting the estimated asymptotic bias from $  \log( \widehat{\COR} )$.
\end{proof}

\section{\bf More Simulation Studies}

In order to compare the finite sample properties of Neyman's original method, the modified Neyman's method, and the Bayesian method, we conduct the following set of simulation studies.
In the main text, we choose the following two sets of potential outcomes: the first set of potential outcomes $(N_{11}, N_{10}, N_{01}, N_{00})$ are independent: $ (50,50,50,50)$, $(30,70,30,70)$, $(30,90,20,60)$, $(80,20,80,20)$, $(60,20,90,30)$; the second set of potential outcomes are positively associated: $(60,40,40,60)$, $(50,50,30,70)$, $(50,70,30,50)$, $(40,110,10,40)$, $(70,30,50,50)$, $(50,30,70,50)$, $(30,10,110,50)$. In addition, in this Supplementary Materials, we also choose negatively associated potential outcomes: $ (40,60,60,40)$, $(30,70,50,50)$, $(40,80,40,40)$, $(30,120,20,30)$, $(50,50,70,30)$, $(40,40,80,40)$, $(20,20,120,40)$. We summarize the ``Science'' in Table \ref{tb::science-appendix}, where Cases 1--5 represent independent potential outcomes, Cases 6--12 represent positively associated potential outcomes, and Cases 13--19 represent negatively associated potential outcomes.

\begin{table}[ht]
\centering
\caption{``Science table'' for the simulation studies}\label{tb::science-appendix}
\begin{tabular}{|c|cccc|cccc|ccc|}
\hline
Case & $N_{11}$ & $N_{10}$ & $N_{01}$ & $N_{00}$ & $S_1^2$ & $S_0^2$ & $S_{10}$ &  $S_\tau^2$ & $\tau$ & $\log(\CRR)$ & $\log(\COR)$ \\
\hline 
1&50&50&50&50&0.251 & 0.251 & 0.000 & 0.503 & 0.000 & 0.000 & 0.000 \\ 
2&30&70&30&70&0.251 & 0.211 & 0.000 & 0.462 & 0.200 & 0.511 & 0.847 \\ 
3&30&90&20&60&0.241 & 0.188 & 0.000 & 0.430 & 0.350 & 0.875 & 1.504 \\ 
4&80&20&80&20&0.251 & 0.161 & 0.000 & 0.412 & -0.300 & -0.470 & -1.386 \\ 
5&60&20&90&30&0.241 & 0.188 & 0.000 & 0.430 & -0.350 & -0.629 & -1.504 \\ 
\hline 
6&60&40&40&60&0.251 & 0.251 & 0.050 & 0.402 & 0.000 & 0.000 & 0.000 \\ 
7&50&50&30&70&0.251 & 0.241 & 0.050 & 0.392 & 0.100 & 0.223 & 0.405 \\ 
8&50&70&30&50& 0.241 & 0.241 & 0.010 & 0.462 & 0.200 & 0.405 & 0.811 \\ 
9&40&110&10&40&0.188 & 0.188 & 0.013 & 0.352 & 0.500 & 1.099 & 2.197 \\ 
10&70&30&50&50&0.251 & 0.241 & 0.050 & 0.392 & -0.100 & -0.182 & -0.405 \\ 
11&50&30&70&50&0.241 & 0.241 & 0.010 & 0.462 & -0.200 & -0.405 & -0.811 \\ 
12&30&10&110&50& 0.161 & 0.211 & 0.010 & 0.352 & -0.500 & -1.253 & -2.234 \\  
\hline  
13&40&60&60&40&0.251 & 0.251 & -0.050 & 0.603 & 0.000 & 0.000 & 0.000 \\ 
14&30&70&50&50&0.251 & 0.241 & -0.050 & 0.593 & 0.100 & 0.223 & 0.405 \\ 
15&40&80&40&40& 0.241 & 0.241 & -0.040 & 0.563 & 0.200 & 0.405 & 0.811 \\
16&30&120&20&30&0.188 & 0.188 & -0.038 & 0.452 & 0.500 & 1.099 & 2.197 \\ 
17&50&50&70&30&0.251 & 0.241 & -0.050 & 0.593 & -0.100 & -0.182 & -0.405 \\ 
18&40&40&80&40&0.241 & 0.241 & -0.040 & 0.563 & -0.200 & -0.405 & -0.811 \\ 
19&20&20&120&40&0.161 & 0.211 & -0.040 & 0.452 & -0.500 & -1.253 & -2.234 \\  
\hline 
\end{tabular}
\end{table}

For given potential outcomes, we draw, repeatedly and independently, the treatment assignment vectors $5000$ times, and apply the three methods after obtaining the observed outcomes.
We compare three methods: Neymanian inference assuming constant treatment effects, improved Neymanian inference, and Bayesian inference assuming independent potential outcomes.

The results are summarized in Figure \ref{fg::independent}, \ref{fg::positive} and \ref{fg::negative}, with average biases, average lengths of the $95\%$ confidence/credible intervals, and the coverage probabilities. The main text only reports the results for $\CRD$ and $\log(\COR)$, here we report the results for all causal measures.
When the potential outcomes are independent or positively associated,
the results for $\log(\CRR)$ are similar to those for $\log(\COR)$ as discussed in the main text.
When the potential outcomes are negatively associated, all the interval estimates over cover the true causal measures, while the Bayesian credible intervals are the narrowest.

\begin{figure}[ht]
\centering
\includegraphics[width = \textwidth]{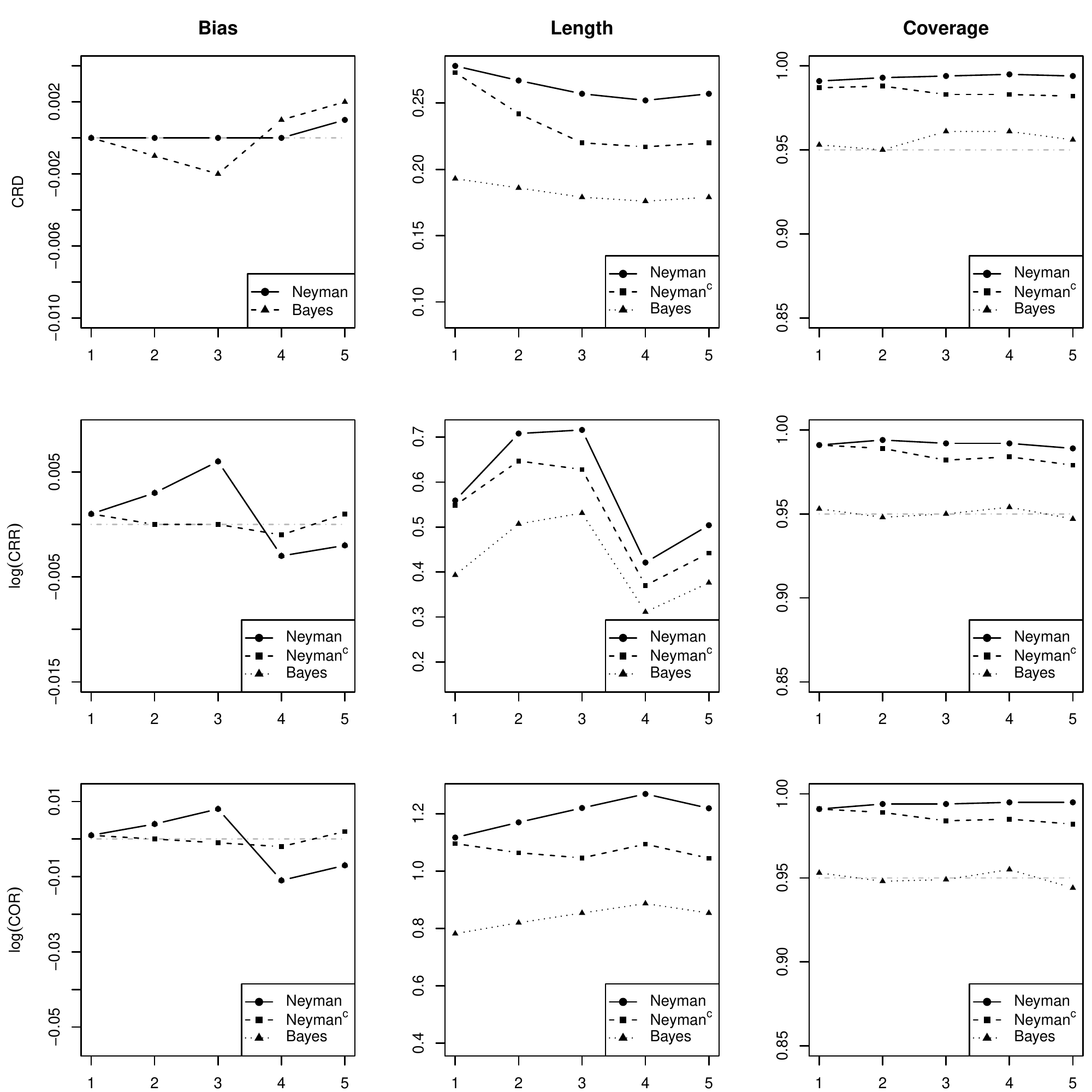}
\caption{Simulation Results for Independent Potential Outcomes. Each subfigure is a $2\times 3$ matrix summarizing $3$ repeated sampling properties (average bias, average length, and coverage of interval estimates) for $2$ causal measures.
Note that ``Neyman'' and ``Bayes'' are indistinguishable for biases of $\log(\CRR)$ and $\log(\COR)$.
} \label{fg::independent}
\end{figure}

\begin{figure}[ht]
\centering
\includegraphics[width = \textwidth]{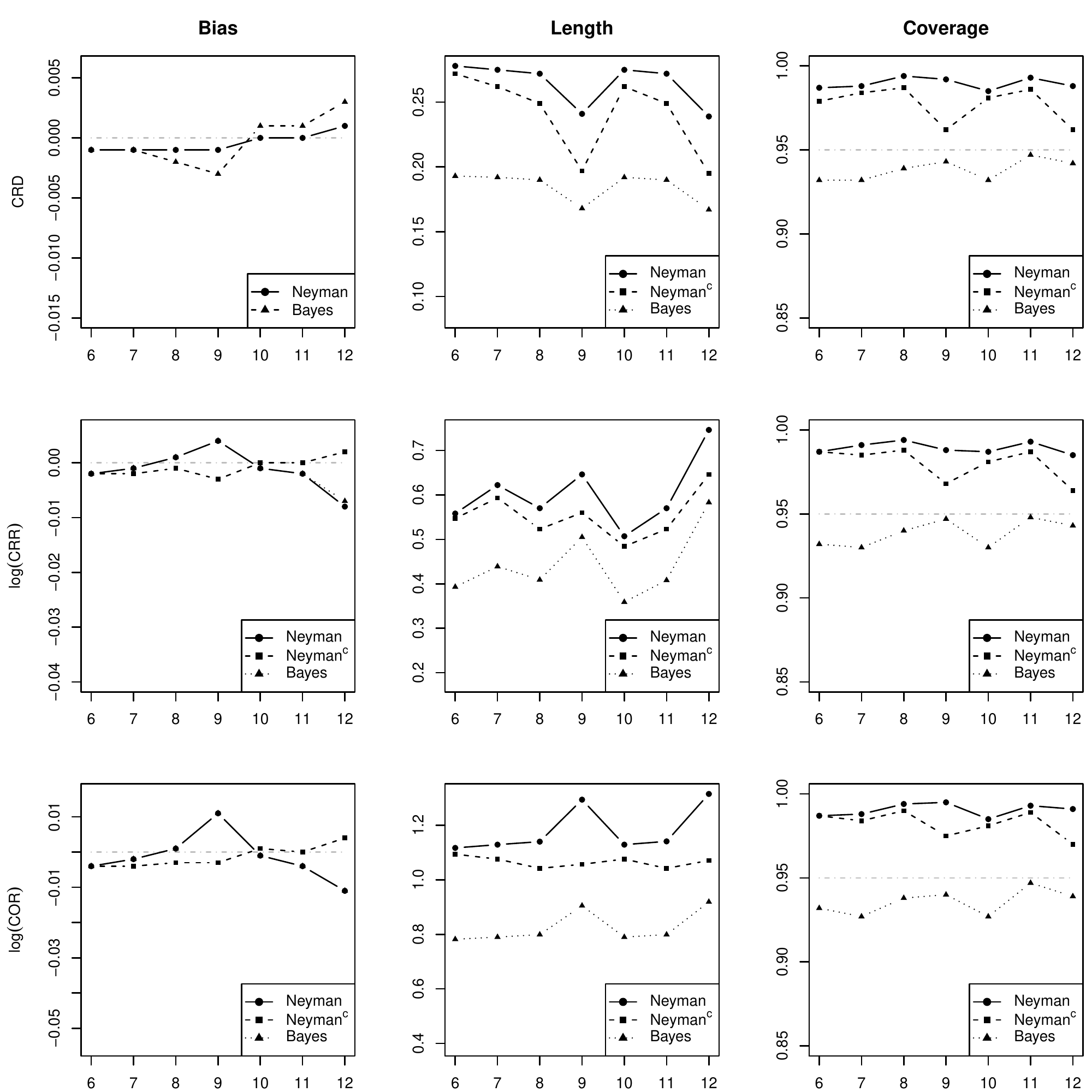}
\caption{Simulation Results for Positively Associated Potential Outcomes. Each subfigure is a $2\times 3$ matrix summarizing $3$ repeated sampling properties (average bias, average length, and coverage of interval estimates) for $2$ causal measures. Note that ``Neyman'' and ``Bayes'' are indistinguishable for biases of $\log(\CRR)$ and $\log(\COR)$.
} \label{fg::positive}
\end{figure}

\begin{figure}[ht]
\centering
\includegraphics[width = \textwidth]{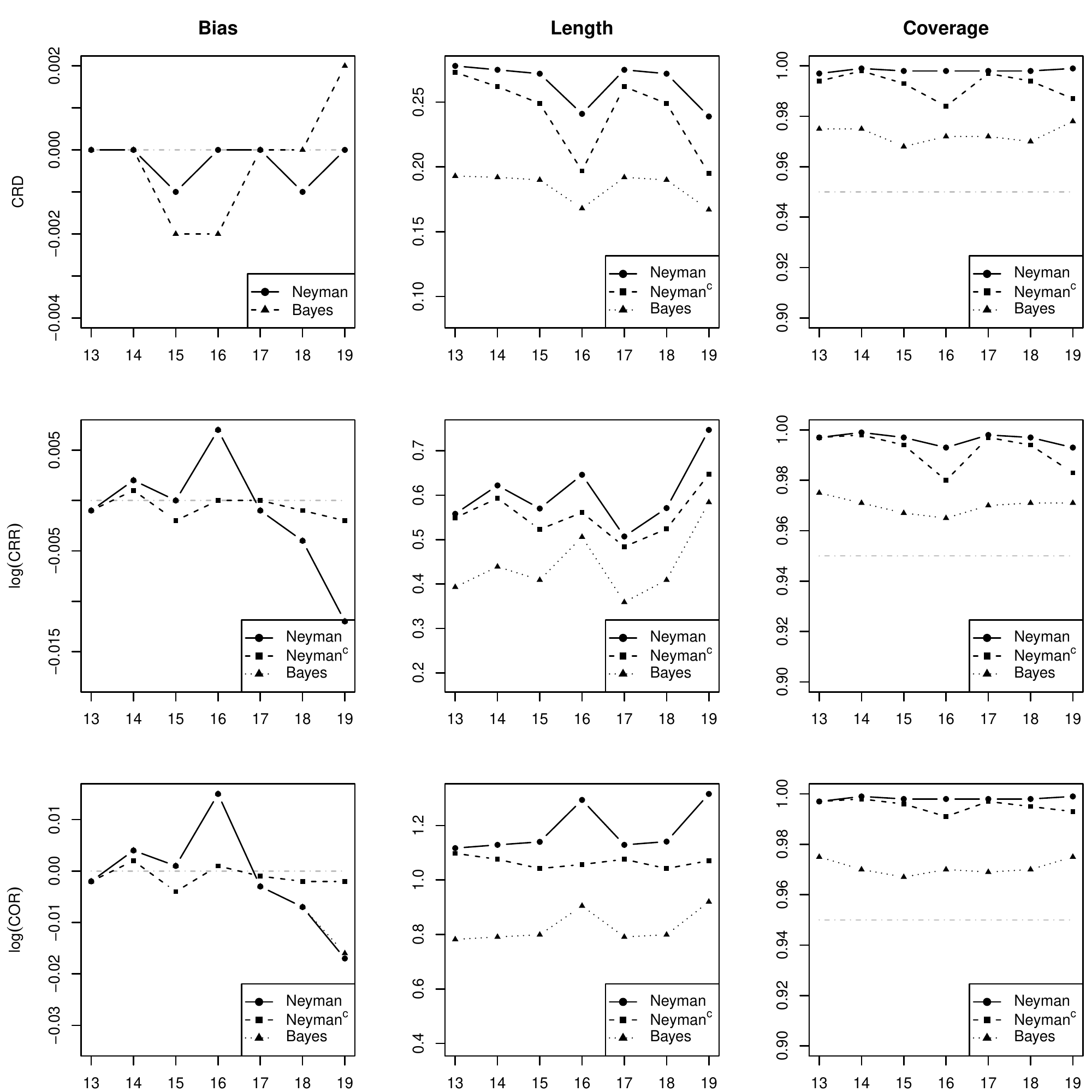}
\caption{Simulation Results for Negatively Associated Potential Outcomes. Each subfigure is a $2\times 3$ matrix summarizing $3$ repeated sampling properties (average bias, average length, and coverage of interval estimates) for $2$ causal measures. Note that ``Neyman'' and ``Bayes'' are indistinguishable for biases of $\log(\CRR)$ and $\log(\COR)$.
} \label{fg::negative}
\end{figure}

\section{\bf  More Details about the Application}

As in the main text, the example is taken from Bissler et al. (2013), and they compare the rate of adverse events in the treatment group versus the control group. The adverse event naspharyngitis occurred in $19$ among $79$ subjects in the treatment group with everolimus, and it occurred in $12$ among $39$ subjects in the control group. Therefore, the $2\times 2$ table representing the observed data has cell counts $(n_{11}, n_{10}, n_{01}, n_{00}) = (19,60,12,27).$
Figure \ref{fg::sense1} shows the sensitivity analysis for $\CRD, \log(\CRR)$ and $\log(\COR)$, with similar patterns for all of them.

\begin{figure}[ht]
\centering
\includegraphics[width = \textwidth]{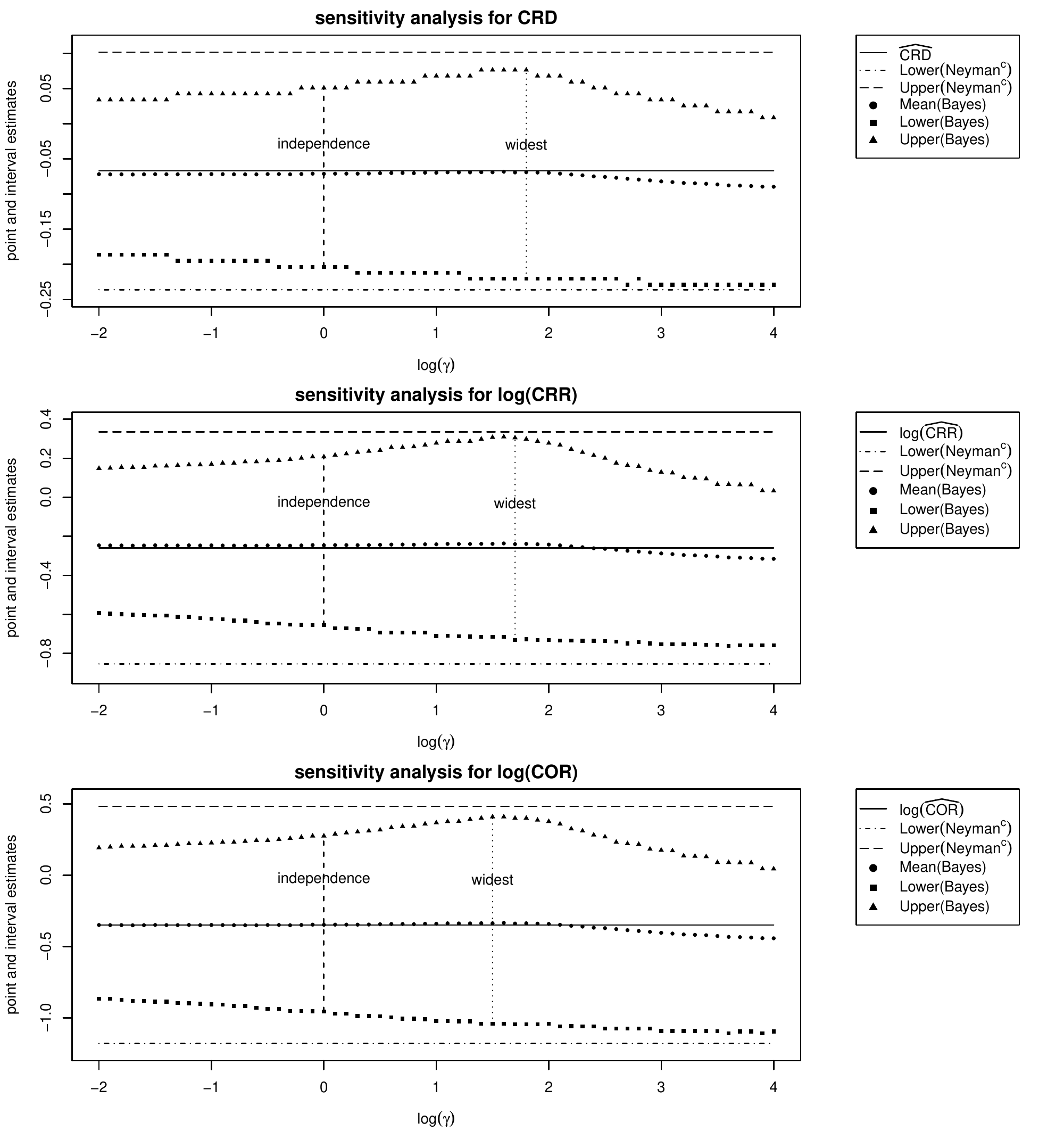}
\caption{Bayesian Sensitivity Analysis of the Trial with $(n_{11}, n_{10}, n_{01}, n_{00}) =  (19,60,12,27)$.
Three panels are for $\CRD, \log(\CRR)$, and $\log(\COR)$, respectively.
The intervals named ``independence'' are the $95\%$ posterior credible intervals under independence of the potential outcomes,  and the intervals named ``widest'' are the widest $95\%$ credible intervals over the ranges of the sensitivity parameters.
} \label{fg::sense1}
\end{figure}

\end{document}